\documentclass{article}

\usepackage{endnotes}
\let\footnote=\endnote

\usepackage{tikz}
\usetikzlibrary{positioning}
\usetikzlibrary{calc}
\usetikzlibrary{fit}

\usepackage{subcaption, multirow}
\usepackage{amsmath, mathtools}

\usepackage{xltabular}
\PassOptionsToPackage{hyphens}{url}\usepackage[hypertexnames=false]{hyperref}

\usepackage{enumitem}

\usepackage{natbib}
 \bibpunct[, ]{(}{)}{,}{a}{}{,}

\usepackage{multibib}
\newcites{appendix}{References}

\usepackage{setspace}
\usepackage{booktabs}

\usepackage[margin=1in]{geometry}
\usepackage{apptools}

\usepackage{amsthm,amssymb}

\newtheorem{theorem}{Theorem}
\newtheorem{lemma}{Lemma}
\newtheorem{proposition}{Proposition}
\newtheorem{corollary}{Corollary}

\newtheorem{remark}{Remark}
\newtheorem{definition}{Definition}

\AtAppendix{\counterwithin{figure}{section}}
\AtAppendix{\counterwithin{table}{section}}
\AtAppendix{\counterwithin{equation}{section}}
\AtAppendix{\renewcommand{\arraystretch}{1}}
\hypersetup{
           breaklinks=true,
           colorlinks=false
        }
        
\usepackage{algorithm}
\usepackage{algpseudocode}
\usepackage{tikz}

\usepackage{graphicx}%
\usepackage{multirow}%
\usepackage{amsmath,amssymb,amsfonts}%
\usepackage{mathrsfs}%
\usepackage[title]{appendix}%
\usepackage{xcolor}%
\usepackage{textcomp}%
\usepackage{manyfoot}%
\usepackage{booktabs}%
\usepackage{algorithm}%
\usepackage{algorithmicx}%
\usepackage{algpseudocode}%
\usepackage{listings}%
\usepackage{diagbox}
\usepackage{dsfont}
\usepackage{hyperref}
\usepackage{tikz}
\usepackage{pgfplots}
\usepackage{xcolor}%
\newcommand{\rev}[1]{{\color{black}#1}}
\usepackage{scalefnt}


\newcommand{\tighttable}{%
  \setlength{\tabcolsep}{3pt}        
  \renewcommand{\arraystretch}{1} 
}

\title{Superadditivity-based valid inequalities and asymptotic bounds for the vehicle routing problem with stochastic demands}
\author{
Robin Legault\textsuperscript{1}\thanks{Corresponding author. \href{mailto:legault@mit.edu}{legault@mit.edu}} \and
Panca Jodiawan\textsuperscript{2} \and
Jean-François Côté\textsuperscript{3,4} \and
Leandro C. Coelho\textsuperscript{3,4}
}

\date{
\normalsize{
\textsuperscript{1}Operations Research Center, Massachusetts Institute of Technology, Cambridge, MA, USA\\
\textsuperscript{2}Department of Industrial Engineering and Management, Yuan Ze University, Taoyuan City, Taiwan\\
\textsuperscript{3}CIRRELT and Faculté des Sciences de l’Administration, Université Laval, Québec, Canada\\
\textsuperscript{4}Canada Research Chair in Integrated Logistics, Université Laval, Québec, Canada
}
}

\begin{document}

\maketitle

\begin{abstract}
Over the past thirty years, the vehicle routing problem with stochastic demands (VRPSD) has emerged as a canonical application of the integer L-shaped method. Recently, the disaggregated integer L-shaped (DL-shaped) method, which decomposes the recourse function by customer rather than treating it as an aggregate cost, has been proposed for the VRPSD under the classical detour-to-depot policy. However, its generalizability to other recourse policies has not been investigated. In this work, we identify the property that characterizes the validity of the DL-shaped reformulation: the superadditivity of the recourse function under path concatenation. We show that superadditivity holds under the optimal restocking policy, and rectify an incorrect argument from the original paper on the DL-shaped method, rigorously establishing its validity under the detour-to-depot policy. We then introduce a new family of valid inequalities, \rev{the edge-set cuts, which generalize the original DL-shaped cuts and are analytically shown to provide structural advantages over existing inequalities}. Building on these results, we develop a DL-shaped algorithm for the VRPSD with optimal restocking. Our algorithm outperforms existing methods in the high customer-to-vehicle ratio regime and solves 14 open single-route instances. \rev{We further derive asymptotic bounds on the optimal value of the VRPSD in a Euclidean setting with i.i.d.\ customers. This analysis reveals an asymptotic equivalence between the VRPSD and the split-delivery vehicle routing problem. It also yields tight bounds on the cost of requiring the total expected demand on each route not to exceed the vehicle capacity, resolving an open question in our asymptotic setting.}
\end{abstract}

\textit{Keywords: } Vehicle routing problem, Stochastic demands, Optimal restocking, Integer L-shaped method, Superadditivity

\section{Introduction} \label{sec:Intro}

Since the seminal work of \cite{tillman1969multiple}, the vehicle routing problem with stochastic demands (VRPSD) has been extensively studied \cite{oyola2017stochastic, oyola2018stochastic}. \rev{The VRPSD consists in designing routes for capacity-limited vehicles to serve customers with stochastic demands at minimum expected cost, where each demand realization is revealed only when the vehicle arrives at the corresponding customer. This paper considers the VRPSD in its classical form: homogeneous vehicles, a single depot, a single commodity, and no time windows.}

We study the problem under the \textit{a priori} optimization framework \cite{bertsimas1990priori}, in which routes are chosen before the demand realizations are observed \rev{and vehicles may perform restocking trips to the depot as they execute their routes. The available actions for managing capacity shortages are determined by a recourse policy. Under the classical \textit{detour-to-depot} (DTD) policy \cite{dror1989vehicle}, a vehicle returns to the depot only when a \textit{failure} occurs, that is, when the realized demand of the current customer exceeds its residual capacity. Under the \textit{optimal restocking} (OR) policy \cite{yee1980note, yang2000stochastic}, a vehicle may also perform a \textit{preventive} restocking trip to the depot between two consecutive customers.}

\paragraph{Exact algorithms for the VRPSD.} A rich body of exact algorithms has been developed to solve the VRPSD under the DTD and OR policies. These algorithms can be divided into integer L-shaped methods \cite{gendreau1995exact, hjorring1999new, laporte2002integer, jabali2014partial, louveaux2018exact, salavati2019exact, hoogendoorn2023improved, parada2024disaggregated, hoogendoorn2025evaluation} and branch-price-and-cut (BP\&C) methods \cite{christiansen2007branch, gauvin2014branch, florio2020new, florio2023recent}. Integer L-shaped approaches model the VRPSD as a two-stage stochastic program and rely on edge-flow or arc-flow formulations. Within a branch-and-cut (B\&C) framework, they iteratively generate valid inequalities that tighten the approximation of the recourse function and characterize feasible routes. On the other hand, BP\&C methods embed column generation inside a branch-and-bound (B\&B) tree. The master problem is a set-partitioning model whose columns represent complete routes. Each iteration solves a resource-constrained shortest-path problem \cite{irnich2005shortest} to identify a new route with negative reduced costs.

Although both integer L-shaped and BP\&C algorithms operate within a B\&B framework, their computational bottlenecks differ fundamentally, resulting in complementary performance profiles. In integer L-shaped algorithms, most of the effort is spent repeatedly solving the linear programming (LP) relaxation of the relaxed master problem, which becomes increasingly complex as more cuts are generated. Computational efficiency therefore relies on devising valid inequalities that rapidly tighten the master problem, thereby limiting the number of explored nodes. In comparison, the main computational burden in BP\&C methods lies in solving an NP-hard pricing problem \cite{ota2025hardness} at each iteration. In state-of-the-art BP\&C algorithms, this pricing subproblem is solved via a labeling algorithm, where each label represents a partial path from a customer back to the depot. The worst-case number of labels generated grows exponentially with the maximum number of customers that can be visited on a route, and the computational cost of extending each label scales linearly with both the vehicle capacity and the cardinality of the random demand supports. Recent developments in BP\&C methods for the VRPSD largely focus on techniques designed to limit label proliferation, notably dominance rules and reduced-cost bounds. The BP\&C methods of \cite{florio2020new, florio2023recent}, which leverage effective labeling strategies and benefit from the tight LP relaxation provided by the set-partitioning formulation of the VRPSD, achieve state-of-the-art results on instances with many short routes and moderate vehicle capacities. However, when optimal solutions consist of only a few routes, the decomposition benefit of column generation weakens significantly, as each column encapsulates a substantial portion of the decisions that form a feasible solution. Notably, when the solution comprises a single route, the set-partitioning formulation collapses to solving a single resource-constrained shortest-path problem. Reflecting a structural similarity with the traveling salesman problem (TSP), where the best-performing exact methods are based on edge-flow formulations strengthened by cutting planes \cite{applegate2006traveling}, integer L-shaped algorithms remain the state of the art for instances involving a small number of long routes.

\paragraph{Integer L-shaped algorithms.} Traditional integer L-shaped methods model recourse costs using a single auxiliary variable, which is progressively tightened by valid inequalities. In the VRPSD literature, these inequalities have typically been categorized as either \textit{optimality cuts} or \textit{lower bounding functionals} (LBFs), although both types can be described as optimality cuts from a broader Benders decomposition perspective \cite{hooker2023logic}. Optimality cuts denote valid inequalities that are tight at the current solution or explicitly exclude the current solution from the feasible domain, whereas LBFs are more general lower bounds that can be \textit{active}, \rev{i.e., impose a strictly positive bound on the recourse cost of a first-stage solution}, over a larger subset of the feasible domain. Optimality cuts alone are sufficient to ensure convergence to an optimal solution. However, they are typically active for a single integer feasible solution and thus do not substantially contribute to improving the global lower bound.

The significant progress of integer L-shaped methods over the last thirty years is principally attributable to the development of LBFs that either provide tighter bounds or apply to broader classes of solutions. Most contributions along these lines are based on the notion of a \textit{partial route}, which corresponds to a sequence of ordered and unordered sets of nodes that starts and ends with singletons comprising only the depot. Any route that visits the same customers and respects their partial ordering is then said to \textit{adhere} to the partial route. The \textit{partial route inequalities} were introduced by \cite{hjorring1999new} for the single-route VRPSD under the DTD policy, extended to multiple routes by \cite{laporte2002integer}, generalized by \cite{jabali2014partial}, and adapted to the OR policy by \cite{louveaux2018exact} and \cite{salavati2019exact}. Recently, \cite{hoogendoorn2023improved, hoogendoorn2025evaluation} departed from the common practice of representing the overall recourse cost by a single auxiliary variable. They developed the \textit{partial route-split} and \textit{multi-route-split} inequalities, which decompose the recourse cost into individual routes. These LBFs impose a lower bound that is enforced whenever the solution contains routes that adhere to each partial route in a collection.

Subsequently, \cite{parada2024disaggregated} proposed further decomposing the recourse cost by customer. They introduced new optimality cuts and LBFs called \textit{path cuts} (P-cuts) and \textit{set cuts} (S-cuts), which remain active across all solutions featuring specific paths or sets of consecutively visited customers. This contrasts with all the optimality cuts and partial route inequalities previously described in the literature, which are only active when a solution contains a set of routes that form a predefined partition of a given set of customers. The resulting B\&C algorithm, called the \textit{disaggregated integer L-shaped} (DL-shaped) method, empirically dominates the previous integer L-shaped methods under the DTD policy. However, in addition to relying on policy-specific lower bounds, the DL-shaped method requires specific conditions on the recourse function to be a valid exact algorithm. These conditions have only been developed for the DTD policy, leaving the generalizability of the DL-shaped method as an open question. More fundamentally, as we will show, the conditions have not been adequately characterized. We summarize our contributions as follows.

\paragraph{Contributions.}
\begin{enumerate}
    \item \textbf{Analysis of the DL-shaped method:} We provide a formal and general presentation of the DL-shaped reformulation of the VRPSD. We show that superadditivity of the recourse function under path concatenation is the structural property that characterizes the validity of this reformulation.
    \item \textbf{Recourse properties:} We study the properties of the DTD and OR recourse functions. We exhibit an incorrect argument in the proof of the main result of \cite{parada2024disaggregated}, and show that the validity of their results can be rigorously established based on superadditivity. We then show that the superadditivity property is always satisfied under the OR policy.
    \item \textbf{New valid inequalities:} We introduce the \textit{edge-set cuts} (E-cuts), a new family of DL-shaped cuts generalizing the P-cuts and the S-cuts. \rev{We study their relation to partial route-split inequalities and exhibit a family of instances where a single E-cut suffices for convergence while arbitrarily many existing valid inequalities are required. We propose an edge selection strategy for the E-cuts and develop lower bounds on the OR recourse function that enable their efficient instantiation under this policy.}
    \item \textbf{Algorithm and computational study:} Building on these new inequalities, we develop a DL-shaped algorithm for the VRPSD under the OR policy. Computational experiments on standard and new instances show that our algorithm outperforms existing methods \rev{on instances with a high customer-to-vehicle ratio}. In the single-route case, we increase the number of solved instances across two benchmark sets from 2 to 16.
    \item \rev{\textbf{Asymptotic analysis:} We provide an asymptotic analysis of the optimal value of the VRPSD in a Euclidean setting with i.i.d.\ customers. Under both the DTD and OR policies, we show that the asymptotic per-customer cost of the VRPSD matches that of the deterministic split-delivery VRP, and that requiring the total expected demand on each route not to exceed the vehicle capacity inflates this rate by a factor determined by a bin-packing constant. In particular, this factor is at most 3, which resolves an open question of \cite{hoogendoorn2025evaluation} in our asymptotic setting.}
\end{enumerate}
\rev{\paragraph{Outline.} The remainder of the paper is organized as follows. Section~\ref{sec:ProbForm} introduces the VRPSD. Section~\ref{sec:analysis_DL_shaped} presents the DL-shaped reformulation and characterizes the superadditivity condition for its validity. Section~\ref{sec:properties} studies the DTD and OR recourse functions. Section~\ref{sec:E-cuts} introduces the E-cuts. Section~\ref{sec:implementation} describes our DL-shaped algorithm, and Section~\ref{sec:experiments} reports the computational results. Section~\ref{sec:asymptotics} presents the asymptotic analysis. Section~\ref{sec:Conc} concludes.}

\section{Problem definition} \label{sec:ProbForm}

Section~\ref{subsec:notation} introduces our setup and basic definitions. \rev{Section~\ref{subsec:assumptions} states the modeling assumptions. Section~\ref{subsec:limitations} delineates the scope of the work and highlights the limitations of our framework.} Section~\ref{subsec:VRPSDform} presents the mathematical formulation of the VRPSD and defines the problem variants we study. Section~\ref{subsec:ExpRecCost} defines the recourse function under the OR and DTD policies.

\subsection{Setup and definitions}\label{subsec:notation}

\paragraph{Graph.} The problem is defined on a complete undirected graph $G := (N_0, E)$ with nodes $N_0 := \{0\} \cup N$, where $0$ represents the depot and $N := \{1, \ldots, n\}$ is the set of customers. The edge set is $E := \{\{i,j\} : i, j \in N_0,\, i \neq j\}$, and each edge $e \in E$ is assigned a non-negative traveling cost $c_e \rev{\in \mathbb{R}_{\geq 0}}$. These costs satisfy the triangle inequality. For any subset $S \subseteq N$, we write $E(S) := \{\{i,j\} \in E : i, j \in S\}$ for the edges with both endpoints in~$S$. For any node $i \in N_0$, we write $\delta(i) := \{e \in E : i \in e\}$ for the edges incident to~$i$, and for disjoint sets $A, B \subset N_0$, $\delta(A,B) := \{\{i,j\} \in E : i \in A,\, j \in B\}$.
 
\paragraph{Sequences of customers.} A \textit{path} is a sequence $p = (i_1, \dots, i_t)$ of distinct customers; we denote by $N(p) := \{i_1, \dots, i_t\} \subseteq N$ the set of nodes it visits. A sequence $p' = (j_1, \dots, j_{t'})$ is a \textit{subsequence} of~$p$ if there exist indices $1 \leq a_1 < \dots < a_{t'} \leq t$ such that $j_k = i_{a_k}$ for all $k \in \{1, \dots, t'\}$. For $1 \leq a \leq b \leq t$, the sequence $(i_a, \dots, i_b)$ is a \textit{subpath} of~$p$. A subsequence is thus obtained by removing arbitrary elements from a path, whereas a subpath is obtained by removing elements only from its beginning and end. We say that a path~$p'$ is a subsequence (respectively subpath) of a set of paths~$\bar{\mathcal{P}}$ if $p'$ is a subsequence (respectively subpath) of some $p \in \bar{\mathcal{P}}$.

\rev{
\subsection{Modeling assumptions}\label{subsec:assumptions}

\paragraph{Locations and demands.}
Customer locations are known, while demands are modeled as independent, non-negative random variables $\{\xi_i\}_{i \in N}$ with known distributions, supports $\{\Xi_i\}_{i \in N}$, and finite means $\{\mu_i\}_{i \in N}$. The realization of the demand~$\xi_i$ is revealed only upon the vehicle's arrival at customer~$i$. In general, no restrictions are imposed on the supports beyond non-negativity. Each $\Xi_i \subseteq \mathbb{R}_{\geq 0}$ may be discrete or continuous, and bounded or unbounded. For the OR policy, we assume in the algorithmic sections of the paper that the demands are integer-valued, i.e., $\Xi_i \subseteq \mathbb{Z}_{\geq 0}$, and denote the probability masses by $\rho^s_i := \mathbb{P}[\xi_i = s]$ for each $s \in \Xi_i$. This allows the recourse function to be expressed through finite-state dynamic programming. The OR policy remains well-defined for general non-negative demand distributions with finite means. 
 
\paragraph{Fleet and routes.}
We consider a single depot and a fleet of identical vehicles, each with capacity~$Q \in \mathbb{R}_{> 0}$. A set of routes, each forming a simple cycle on~$G$ starting and ending at the depot, must be selected, with the requirement that each customer $i \in N$ appears in exactly one route. Each route is permitted to cover a set of customers whose total expected demand does not exceed~$fQ$, where $f \in \mathbb{R}_{>0}$ is called the \textit{load factor}. For feasibility, we assume $\max_{i \in N}\{\mu_i\} \leq fQ$, so that each customer can be assigned to a route. In addition, one may impose explicit restrictions on the number of routes that can appear in a solution. We denote by $M \subseteq \mathbb{N}$ the set of admissible route counts, with $|M|=1$ if the number of routes is fixed, and $M = \{1, \dots, n\}$ when the number of routes is not restricted beyond the load factor requirement. 
 
\paragraph{Recourse structure.}
We consider \emph{route-based} recourse policies~\cite{dror1986stochastic}, in which vehicles serve disjoint sets of customers and the recourse actions available to a vehicle depend only on the demands observed along its own route. Under this assumption, the total expected recourse cost decomposes additively by route.
 
\paragraph{Recourse cost parameters.}
A \emph{failure} occurs when the residual capacity of a vehicle does not suffice to serve a customer's demand; a \emph{preventive return} is a voluntary round trip to the depot to restock before visiting the next customer. For the sake of generality, we assume that in addition to travel distance, fixed penalty terms $b^F \geq 0$ and $b^P \geq 0$ with $b^P \leq b^F$ can be applied for failures and preventive returns, respectively. For each customer $i \in N$, the failure cost is then $c^F_i := b^F + 2\,c_{0,i}$. For each pair of consecutively visited customers $i,j \in N$, the preventive return cost is $c^P_{i,j} := b^P + c_{0,i} + c_{0,j} - c_{i,j}$.

\subsection{Scope and limitations}\label{subsec:limitations}
 
Two assumptions deserve further comment. First, the independence of the demands is essential to our framework. It implies that, conditional on the demands observed along a route, the distribution of the remaining unserved demands coincides with their marginal distribution. This property is central to the proof of superadditivity of the OR recourse function (Theorem~\ref{thm:Q_superadditive}), where a concatenated path is decomposed into a prefix and a suffix, and the cost incurred on the suffix is bounded below by the OR recourse cost of that suffix as a standalone path. Under correlated demands, the information revealed along the prefix may help guide restocking decisions later in the route, so this lower-bound comparison need not remain valid. Moreover, under correlated demands, the OR recursion would generally need to track not only the residual vehicle capacity but also the conditional distribution of the remaining demands. While it may be possible to identify sufficient conditions that preserve superadditivity in this setting, we leave this direction for future work.

Second, the route-based recourse framework excludes VRPSD settings with coordinated recourse decisions across vehicles, such as the overlapped-routing paradigm of~\cite{ledvina2022new}, in which the same customer can appear on multiple routes, and the recourse assignment depends jointly on all realized demands. In such models, the recourse cost of serving a path or a subset of customers cannot be characterized independently of the rest of the solution. Consequently, the path- and set-indexed coefficients that appear in the DL-shaped cuts are no longer well defined as functions of those objects alone. More fundamentally, the superadditivity property (Definition~\ref{def:superadditivity}) relies on the concatenation of paths and on the recourse function being a well-defined map from individual paths to costs. When routes may overlap, this concatenation operation no longer suffices to express the recourse structure, making route separability a foundational assumption of this work. Generalizing superadditivity to broader mappings from routing solutions to recourse costs is a natural direction for extending the applicability of the DL-shaped methodology.
}

\subsection{Mathematical formulation of the VRPSD} \label{subsec:VRPSDform}
 
We adopt an edge-flow formulation of the problem. For each edge $e \in E$, an integer variable~$x_e$ represents the number of times $e$~is traversed in the a priori solution. For edges $e \in E(N)$ between customers, $x_e$ is binary; for edges $e \in \delta(0)$ incident to the depot, $x_e \in \{0,1,2\}$, with $x_e = 2$ indicating that a vehicle serves a single customer. A binary variable~$z_m$ indicates whether the solution comprises $m$~routes. The objective is to minimize the sum of the first-stage cost (travel distance along the planned routes) and the second-stage cost (expected recourse cost). The VRPSD is formulated as follows:
\begingroup
\allowdisplaybreaks
\begin{alignat}{2}
\min & \sum_{e \in E} c_{e}x_{e} + \mathcal{Q}(x) \label{objective} \\
\text{s.t.} &     \sum_{e \in \delta(0)}x_{e} = 2\sum_{m \in M}mz_{m}, \label{utilizedVeh} \\
    & \sum_{e \in \delta(i)}x_{e} = 2, &  \rev{\forall} i \in N, \label{flow_cons} \\
    & \sum_{e\in E(S)}x_{e} \leq |S| - \left \lceil \frac{1}{fQ}\sum_{i \in S}\mu_{i} \right \rceil, &   \rev{\forall} S \subseteq N,\label{cap_cut} \\
    & \sum_{m \in M} z_{m} = 1, \label{vehassign} \\
    & x_{e} \in \{0, 1, 2\}, & \rev{\forall} e \in \delta(0), \label{x_depot} \\
    & x_{e} \in \{0, 1\}, & \rev{\forall} e \in E(N), \label{x_custs} \\
    & z_{m} \in \{0, 1\}, & \rev{\forall} m \in M. \label{z_val}
\end{alignat}
\endgroup
 
Constraint~\eqref{utilizedVeh} controls the number of routes in the solution. Constraints~\eqref{flow_cons} guarantee that each customer is visited exactly once. The rounded capacity inequalities~\eqref{cap_cut} eliminate subtours and routes with total expected demand exceeding the allowed load factor. Constraint~\eqref{vehassign} selects exactly one number of routes, and constraints~\eqref{x_depot}--\eqref{z_val} define the domain of the variables.
 
\rev{\paragraph{Variants of the VRPSD.}
Formulation~\eqref{objective}--\eqref{z_val} is parameterized by the set~$M$ of allowed number of routes and the load factor~$f$. Following \cite{hoogendoorn2025evaluation}, we distinguish four variants according to whether the \emph{fixed route constraint} (FRC) and the \emph{expected capacity constraints} (ECCs) are enforced. The FRC imposes a specific number of routes ($|M|=1$), and the ECCs require that the total expected demand assigned to each route does not exceed the vehicle capacity ($f=1$). We call these variants \emph{ECC-FRC}, \emph{ECC}, \emph{FRC}, and \emph{basic}, depending on whether both, only the ECCs, only the FRC, or neither of these conditions is enforced. In the ECC and basic variants, the route-count is not restricted ($M=\{1,\dots,n\}$). In the FRC and basic variants, the right-hand side of constraints~\eqref{cap_cut} simplifies to $|S|-1$, so that the rounded capacity inequalities reduce to subtour-elimination constraints. Throughout the paper, we consider the general formulation~\eqref{objective}--\eqref{z_val} unless otherwise specified. 
}

\subsection{Recourse function}\label{subsec:ExpRecCost}

\rev{A solution $(x,z)$ satisfying constraints \eqref{utilizedVeh}--\eqref{z_val}} defines a set of undirected cycles on $G$, each specifying a sequence of distinct customers that may be visited in either orientation. We characterize such a solution comprising $m$ routes as a collection of paths $\mathcal{P}(x) = \{p_1, \dots, p_m\}$, each taken in an arbitrary orientation, where $\{N(p_1), \dots, N(p_m)\}$ forms a partition of $N$. Denote by $\bar{\mathcal{Q}}_{p}$ the expected recourse cost of the route $(0,p,0)$ specified by path $p=(i_1,\dots,i_t)$, and let $\mathcal{Q}_{p} := \min\{\bar{\mathcal{Q}}_{(i_1,\dots,i_t)}, \bar{\mathcal{Q}}_{(i_t,\dots,i_1)}\}$ be its expected recourse cost in its best orientation. The recourse function evaluated at $x$ is $\mathcal{Q}(x) := \sum_{p \in \mathcal{P}(x)}\mathcal{Q}_p$, \rev{which is independent of the choice of orientations in $\mathcal{P}(x)$}. \rev{Next, we present the expected recourse cost of a route under the OR and DTD policies.}
 
\subsubsection{Optimal restocking policy}\label{sec:OR_pol}

\rev{Under the OR policy, after serving a customer, the vehicle may either perform a preventive restocking trip or proceed directly to the next customer. Here, for exact evaluation of the expected recourse cost by finite-state dynamic programming, we assume that the demands and vehicle capacity are integer-valued.} Denote by $\Psi(s,q) := \max\{\lceil(s - q)/Q\rceil,0\}$ the number of restocking trips required to fulfill a demand of~$s$ given residual capacity~$q$. For a route $(0, p, 0) = (i_0, i_{1}, \dots, i_{t}, i_{t+1})$, the expected recourse cost-to-go $F^{p}_{i_{j}}(q)$ when the vehicle leaves node $i_{j}$, for $j \in \{0,\dots,t{-}1\}$, with residual capacity $q \in \{0,\dots,Q\}$, is given by the following Bellman equation:
\begin{equation}
    F^p_{i_j}(q) := \min
    \begin{cases} \label{def:cost-to-go}
        H^p_{i_{j+1}}(q) := \sum\limits_{s \in \Xi_{i_{j+1}}} \left[c_{i_{j+1}}^F\Psi(s,q) + F^p_{i_{j+1}}\left(\Psi(s,q)Q + q -s  \right) \right]\rho_{i_{j+1}}^s, \\
        H^{*p}_{i_j,i_{j+1}} := c_{i_j, i_{j+1}}^P+H^p_{i_{j+1}}(Q),
    \end{cases}
\end{equation}
with the boundary condition $F^p_{i_{t}}(q):=0 \ \forall q \in \{0,\dots,Q\}$. This boundary reflects that, after serving the last customer $i_{t}$, the vehicle travels back to the depot as planned in the a priori solution, which does not incur a recourse cost. The terms $H^p_{i_{j+1}}(q)$ and $H^{*p}_{i_j,i_{j+1}}$ respectively denote the expected cost-to-go of proceeding directly to customer $i_{j+1}$ and of making a preventive return beforehand. Since travel distances respect the triangle inequality, proceeding to the next customer is always preferable when the vehicle has full capacity, i.e., $F^p_{i_j}(Q) = H^p_{i_{j+1}}(Q)$. We denote by $\bar{\mathcal{Q}}^{\text{OR}}_p:=F^p_{i_0}(Q)$ the expected recourse cost of route $(0,p,0)$ under the OR policy.

\begin{remark}\label{rem:OR-continuous}
    The recursion \eqref{def:cost-to-go} generalizes to random demands with continuous supports by replacing the state space $q \in \{0, \dots, Q\}$ by $q \in [0, Q]$ and defining $H^p_{i_{j+1}}(q) := \mathbb{E}\!\left[c^F_{i_{j+1}} \Psi(\xi_{i_{j+1}}, q) + F^p_{i_{j+1}}\big(\Psi(\xi_{i_{j+1}}, q)Q + q - \xi_{i_{j+1}}\big)\right]$. The monotonicity and superadditivity properties of the OR recourse function extend to this setting; Lemma \ref{lemma:monotonicity_cost_to_go} carries over by replacing in its proof the sums over $\Xi_{i_j}$ with the corresponding expectations, and Theorem \ref{thm:Q_superadditive} then applies verbatim.
\end{remark}
 
\subsubsection{Detour-to-depot policy}\label{sec:DTD_pol}
 
Under the DTD policy, \rev{the vehicle restocks only after a failure}. The expected recourse cost $\bar{\mathcal{Q}}^{\text{DTD}}_p$ of route $(0, p, 0) = (i_0, i_{1}, \dots, i_{t}, i_{t+1})$ is defined as:
\begin{equation}\label{rec:DTD}
    \bar{\mathcal{Q}}^{\text{DTD}}_p := \sum_{j=1}^t\sum_{l=1}^{+\infty}\mathbb{P}\left[ \sum_{k=1}^{j-1}\xi_{i_k} \leq lQ < \sum_{k=1}^j\xi_{i_k} \right]c^F_{i_{j}},
\end{equation}
where the factor multiplying each coefficient $c^F_{i_{j}}$ corresponds to the expected number of failures at customer $i_j$.

\rev{
\begin{remark}\label{rem:DTD_restriction_of_OR}
    The DTD policy can equivalently be framed as the restriction of the OR policy in which preventive restocking trips are prohibited, i.e., $H^{*p}_{i_j,i_{j+1}} = +\infty$. In particular, $\bar{\mathcal{Q}}^{\text{OR}}_p \leq \bar{\mathcal{Q}}^{\text{DTD}}_p$ holds for any path $p$.
\end{remark}
}

\section{The disaggregated integer L-shaped method} \label{sec:analysis_DL_shaped}
We start by reviewing the standard integer L-shaped method in Section \ref{sec:analysis_DL_shaped:L-shaped}. We then present the DL-shaped reformulation of the VRPSD in Section \ref{sec:analysis_DL_shaped:DL-shaped}.

\subsection{Integer L-shaped method}\label{sec:analysis_DL_shaped:L-shaped}
A textbook integer L-shaped method for the VRPSD (as described in \cite{laporte1998solving}) can be implemented from the edge-flow formulation \eqref{objective}--\eqref{z_val} by replacing the recourse function $\mathcal{Q}(x)$ by a single auxiliary variable $\Theta \geq 0$. The problem can then be solved through B\&C by \rev{separating, at each node of the B\&B tree, the rounded capacity inequalities \eqref{cap_cut} that are violated at the current solution}, and adding, at each integer node with solution $(x^{\nu}, \Theta^{\nu})$ such that $\Theta^{\nu} < \mathcal{Q}(x^{\nu})$, the classical optimality cut:
\begin{equation} \label{classic_opt_cut}
    \Theta \geq L + (\mathcal{Q}(x^{\nu}) - L) \left( \sum_{\substack{e \in E(N) \\x^{\nu}_{e}=1}}x_{e} - \sum_{\substack{e \in E(N)}}x^{\nu}_{e} + 1 \right),
\end{equation}
where $L \geq 0$ is a global lower bound on the expected recourse. Inequality \eqref{classic_opt_cut} reduces to $\Theta \geq \mathcal{Q}(x^{\nu})$ if $x=x^{\nu}$, and imposes a lower bound of at most $L$ on $\Theta$ for any integer solution $x \neq x^{\nu}$. It is thus a valid lower bound on the recourse function $\mathcal{Q}$ for any feasible first-stage solution, and is tight at $x=x^{\nu}$. Since the feasible domain of the VRPSD has finite cardinality, these properties suffice to guarantee the finite convergence of the integer L-shaped method to an optimal solution. In practice, this cut has several drawbacks. First, it is only valid for a fixed number of vehicles ($|M|=1$), as discussed in \cite{hoogendoorn2025evaluation}. Furthermore, it is active for a single integer solution, which makes it inefficient at improving the LP relaxation of the master problem. To accelerate convergence, practical implementations such as those of \cite{laporte2002integer}, \cite{jabali2014partial}, \cite{louveaux2018exact}, and \cite{salavati2019exact} also generate LBFs. We refer the reader to the recent works of \cite{hoogendoorn2023improved, hoogendoorn2025evaluation} for a complete review and corrected versions of the LBFs implemented in the integer L-shaped algorithms from the literature. 

\subsection{Disaggregated integer L-shaped method}\label{sec:analysis_DL_shaped:DL-shaped}
The DL-shaped method, introduced in \cite{parada2024disaggregated}, decomposes the auxiliary variable $\Theta$ into customer-specific variables $\{\theta_{i}\}_{i \in N}$. We introduce notation, then formulate its master problem. Let $\mathcal{P} := \{\text{$p$ : $\sum_{i \in N(p)} \mu_i \leq fQ$}\}$ denote the set of feasible paths under the vehicle capacity $Q$ and the load factor $f$. For each path $p = (i_1, \dots, i_t) \in \mathcal{P}$, define the \rev{linear functional $p(x) := \sum_{j=1}^{t-1} x_{i_j, i_{j+1}}$, which evaluates the total flow along its edges for any $x \in \mathbb{R}^E$}. For each set of customers $S \subseteq N$, define $\Pi(S,m) := \left\{ \{p_k\}_{k=1}^m \subseteq \rev{\mathcal{P}} : \{N(p_k)\}_{k=1}^m \text{ partitions $S$} \right\}$ as the set of all partitions of $S$ into exactly $m$ feasible paths, $m(S) := \min\{m \in \mathbb{N} : \Pi(S,m) \neq \emptyset\}$ as the minimum number of vehicles required to cover $S$, and $\mathcal{L}(S, m) := \min_{ \pi \in \Pi(S,m)} \sum_{p \in \pi} \mathcal{Q}_{p}$ as the smallest recourse cost achievable by partitioning $S$ into exactly $m$ feasible paths. Let $\mathcal{S} \subseteq \{S \subseteq N\}$ be any collection of subsets of customers, and for each $S \in \mathcal{S}$, let $1 \leq m_S \leq m(S)$ and $0 \leq \mathcal{L}_{S}\leq \mathcal{L}(S, m_S)$ be fixed values. \rev{The DL-shaped reformulation of the VRPSD~\eqref{objective}--\eqref{z_val} yields the following master problem:}
\begingroup
\allowdisplaybreaks
\begin{alignat}{2}
    (\text{MP}) \quad \min & \sum_{e \in E} c_{e}x_{e} + \sum_{i \in N} \theta_{i} \label{DL:obj} \\
\text{s.t.}\ & \text{\eqref{utilizedVeh}--\eqref{z_val}}, \nonumber \\
    & \sum_{i \in N(p)} \theta_{i} \geq \mathcal{Q}_p  \left( p(x)-|N(p)|+2 \right), & \rev{\forall} p \in \mathcal{P}, \label{DL:p-cut}  \\
    & \sum_{i \in S} \theta_{i} \geq \mathcal{L}_{S}\left( \sum_{e \in E(S)}x_{e} - |S| + m_S + 1 \right), &\ \rev{\forall} S \in \mathcal{S},\label{DL:s-cut} \\
    & \theta_{i} \geq 0, & \rev{\forall} i \in N. \label{DL:theta}
\end{alignat}
\endgroup

The P-cut \eqref{DL:p-cut} associated with a path $p \in \mathcal{P}$ imposes that the total recourse cost attributed to the customers of $N(p)$ be at least the recourse cost $\mathcal{Q}_p$ of route $(0,p,0)$ if $p$ is a subpath of $\mathcal{P}(x)$. By contrast, the S-cut \eqref{DL:s-cut} associated with a set $S \in \mathcal{S}$ imposes that the total recourse cost attributed to the customers of $S$ be at least $\mathcal{L}_S$ when there exists $m_S$ subpaths of $\mathcal{P}(x)$ that partition the customers of $S$. The P-cuts and the S-cuts respectively act as optimality cuts and LBFs. In Theorem \ref{thm:validity_DL-shaped}, we characterize the validity of the DL-shaped reformulation based on the following superadditivity property. 

\begin{definition}[Superadditivity property]\label{def:superadditivity}
    An instance of the VRPSD satisfies the superadditivity property if $\mathcal{Q}_{(p_1,p_2)}\geq\mathcal{Q}_{p_1}+\mathcal{Q}_{p_2} \ \forall\,p_1,p_2\in\mathcal P$ such that $(p_1,p_2) \in \mathcal{P}$.
\end{definition}

\begin{theorem}\label{thm:validity_DL-shaped}
\rev{The DL-shaped model MP is a valid reformulation of \eqref{objective}--\eqref{z_val} for every $M \subseteq \{1,\dots,n\}$} if and only if the superadditivity property holds.
\end{theorem}

\begin{proof}
\rev{Necessity is shown by considering the unrestricted route-count case $M=\{1,\ldots,n\}$, where there exists a feasible solution for which the DL-shaped reformulation overestimates the recourse cost whenever the superadditivity property does not hold.} We then show that the superadditivity property is sufficient for MP to be a valid reformulation of \eqref{objective}--\eqref{z_val}.

\paragraph{\textcolor{black}{Necessity.}}\rev{Consider the case $M = \{1, \dots, n\}$, i.e., the number of routes is not restricted beyond the load factor requirement.} Assume the recourse function is not superadditive, i.e., there exist two paths \(p',p''\in\mathcal P\) whose concatenation \(p_1:=(p',p'')\in\mathcal P\) satisfies $\mathcal{Q}_{p_1} < \mathcal{Q}_{p'}+\mathcal{Q}_{p''}$. Consider the first-stage solution $(x^{\nu}, z^{\nu})$ \rev{that comprises route $(0, p_1, 0)$ and assigns every remaining customer to its own route, i.e., $\mathcal{P}(x^{\nu}) = \{p_1\}\cup \{(i) : i \in N\setminus N(p_1)\}$. This solution is feasible since $p_1 \in \mathcal{P}$ by assumption, each singleton path also respects the load factor as $\mu_i \leq fQ$ for every $i \in N$, and the resulting number of routes $m = n + 1 - |N(p_1)|$ lies in $M$. Its recourse cost is $\mathcal{Q}(x^{\nu}) = \sum_{p \in \mathcal{P}(x^{\nu})}\mathcal{Q}_p = \mathcal{Q}_{p_1} + \sum_{i\in N{\setminus}N(p_1)} \mathcal{Q}_{(i)}$.}

Since $p'(x^{\nu})=|N(p')|-1$ and $p''(x^{\nu})=|N(p'')|-1$ by construction, \rev{the factor multiplying the coefficient $\mathcal{Q}_p$ on the right-hand side of \eqref{DL:p-cut} evaluates to $1$ at $x=x^{\nu}$ for each path $p \in \{p', p''\}$. For the first-stage solution $(x^{\nu}, z^{\nu})$, summing the P-cuts of paths $p'$ and $p''$ yields the following lower bound on the auxiliary variables $\{\theta_i\}_{i \in N(p_1)}$:}
\begin{equation}\label{eq:thetaLB_p}
    \sum_{i\in N(p_1)}\theta_i
      = \sum_{i\in N(p')}\theta_i + \sum_{i\in N(p'')}\theta_i
      \geq \mathcal{Q}_{p'} + \mathcal{Q}_{p''} > \mathcal{Q}_{p_1}.
\end{equation}

\rev{For any singleton path $p = (i)$, $i \in N$, we have $p(x)=0$ and $|N(p)|=1$, so the P-cut \eqref{DL:p-cut} simplifies to the unconditional bound $\theta_i \geq \mathcal{Q}_{(i)}$. By adding inequality \eqref{eq:thetaLB_p} and the P-cuts of every singleton path $p \in \{(i) : i \in N\setminus N(p_1)\}$, we obtain that the recourse cost} of the first-stage solution $(x^{\nu},z^{\nu})$ is strictly higher for MP than for the original formulation \eqref{objective}--\eqref{z_val}:
\begin{equation*}
\sum_{i \in N} \theta_i = \rev{\sum_{i\in N(p_1)}\theta_i + \sum_{i\in N{\setminus}N(p_1)}\theta_i > \mathcal{Q}_{p_1} + \sum_{i\in N{\setminus}N(p_1)} \mathcal{Q}_{(i)}} = \mathcal{Q}(x^{\nu}).
\end{equation*}

\paragraph{Sufficiency.} Assume the recourse function is superadditive.
Let $(x^{\nu},z^{\nu})$ be an integer-feasible solution of the first-stage
constraints \eqref{utilizedVeh}--\eqref{z_val}, and let $\mathcal{P}(x^{\nu}) = \{p_1, \dots, p_m\} \in \Pi(N,m)$ be the set of paths it forms. The original formulation’s recourse evaluates to:
\begin{equation}\label{eq:VRPSD_recourse}
    \mathcal{Q}(x^{\nu}) = \sum_{j=1}^{m}\mathcal{Q}_{p_j}.
\end{equation}

For every path $p_j \in \mathcal{P}(x^{\nu})$, the P-cut \eqref{DL:p-cut} reads
$\sum_{i\in N(p_j)}\theta_i \ge \mathcal{Q}_{p_j}$. Summing over $j=1,\dots,m$ gives the lower bound:
\begin{equation}\label{eq:DL_LB}
    \sum_{i\in N}\theta_i \geq \mathcal{Q}(x^{\nu}).
\end{equation}

It follows from \eqref{eq:VRPSD_recourse} and \eqref{eq:DL_LB} that the solution $(x^{\nu},z^{\nu})$ yields an objective value that is not smaller for MP than for the original formulation. To conclude the proof, it remains to exhibit an assignment $\{\theta^{\nu}_i\}_{i \in N}$ of the auxiliary variables that satisfies constraints \eqref{DL:p-cut}--\eqref{DL:theta} and sums to $\mathcal{Q}(x^{\nu})$. 

For each path $p_j=(i^{j}_{1},\dots,i^{j}_{t_j}) \in \mathcal{P}(x^{\nu})$, we set each $\theta^{\nu}_{i^j_k}$, $k \in \{1,\dots,t_j\}$, to the incremental contribution of customer $i^j_k$ to the expected recourse of the partial path $(i^j_{1},\dots,i^j_{k-1})$:
\[
    \theta^{\nu}_{i^{j}_{k}} = \Delta_{i^{j}_{k}}
      := \mathcal{Q}_{(i^{j}_{1},\dots,i^{j}_{k})}
         -\mathcal{Q}_{(i^{j}_{1},\dots,i^{j}_{k-1})},
\]
where $\mathcal{Q}_{()} = 0$. Superadditivity implies that each incremental contribution satisfies $\Delta_{i^{j}_{k}} \geq 0$, and the non-negativity constraints \eqref{DL:theta} are thus respected. \rev{Summing the auxiliary variables of the customers of path $p_j$} yields:
\[
    \sum_{i\in N(p_j)}\theta^{\nu}_i
      = \sum_{k=1}^{t_j}\Delta_{i^{j}_{k}}
      \rev{= \sum_{k=1}^{t_j} \left( \mathcal{Q}_{(i^{j}_{1},\dots,i^{j}_{k})}
         -\mathcal{Q}_{(i^{j}_{1},\dots,i^{j}_{k-1})} \right) } 
      = \mathcal{Q}_{p_j},
\]
and we indeed verify that $\sum_{i \in N} \theta^{\nu}_i = \sum_{j=1}^m \mathcal{Q}_{p_j} = \mathcal{Q}(x^{\nu})$. From there, we show that no P-cut \eqref{DL:p-cut} is violated at $(x,\theta) = (x^{\nu},\theta^{\nu})$. Consider an arbitrary path $p \in \mathcal{P}$. First, if there is no path $p_j \in \mathcal{P}(x^{\nu})$ of which $p$ is a subpath, then $p(x^{\nu})\leq |N(p)|-2$. The right-hand side of the P-cut is then non-positive at $x=x^{\nu}$, and the P-cut is trivially satisfied. Otherwise, $p=(i^j_a,\dots,i^j_b)$ for some $j \in \{1,\dots,m\}$ and $1 \leq a \leq b \leq t_j$, and the right-hand side is $\mathcal{Q}_p$. The left-hand side evaluates to:
\begingroup
\allowdisplaybreaks
\begin{align*}
    \sum_{i \in N(p)} \theta^{\nu}_i &= \sum_{k=a}^b \Delta_{i^{j}_{k}}\\
    &= \sum_{k=a}^b \left(\mathcal{Q}_{(i^{j}_{1},\dots,i^{j}_{k})} -\mathcal{Q}_{(i^{j}_{1},\dots,i^{j}_{k-1})}\right)\\
    &= \mathcal{Q}_{(i^{j}_{1},\dots,i^{j}_{b})} -\mathcal{Q}_{(i^{j}_{1},\dots,i^{j}_{a-1})} \\
    &\geq \mathcal{Q}_{p},
\end{align*}
\endgroup
where the inequality applies superadditivity to the split $(i^{j}_{1},\dots,i^{j}_{a-1}), (i^{j}_{a},\dots,i^{j}_{b})$, which gives:
\begin{align*}
    & \mathcal{Q}_{(i^j_1, \dots, i^j_b)} \geq \mathcal{Q}_{(i^{j}_{1},\dots,i^{j}_{a-1})} + \mathcal{Q}_{(i^{j}_{a},\dots,i^{j}_{b})}\\
    \iff & \mathcal{Q}_{(i^j_1, \dots, i^j_b)} - \mathcal{Q}_{(i^{j}_{1},\dots,i^{j}_{a-1})} \geq \mathcal{Q}_{(i^{j}_{a},\dots,i^{j}_{b})} \\
    \iff & \mathcal{Q}_{(i^j_1, \dots, i^j_b)} - \mathcal{Q}_{(i^{j}_{1},\dots,i^{j}_{a-1})} \geq \mathcal{Q}_{p}.
\end{align*}

Therefore, all the P-cuts \eqref{DL:p-cut} are satisfied. It remains to show that the S-cuts \eqref{DL:s-cut} are respected at $(x,\theta) = (x^{\nu},\theta^{\nu})$. Consider any set $S \in \mathcal{S}$. The factor multiplying $\mathcal{L}_{S}$ in the S-cut can be non-positive at $x = x^{\nu}$, in which case the cut is trivially respected. Otherwise, by definition of $m_S$, this factor is exactly one, meaning that the customers of $S$ are covered by exactly $m_S$ paths of $\mathcal{P}(x^{\nu})$, and appear consecutively in each of these paths. In this case, we denote by $\{\bar{p}_{j_k}\}_{k=1}^{m_S} \in \Pi(S,m_S)$ the subpaths that form a partition of $S$ in the solution, where for each $k \in \{1,\dots,m_S\}$, $j_k \in \{1,\dots,m\}$ is the index of the path of $\mathcal{P}(x^{\nu})$ of which $\bar{p}_{j_k}$ is a subpath. By construction, $\bar{p}_{j_k}(x^{\nu})=|N(\bar{p}_{j_k})|-1$ for each of these subpaths, so that summing their P-cuts at $x=x^{\nu}$ gives:
\begin{align*}
    \sum_{i \in S}\theta_i = \sum_{k=1}^{m_S} \sum_{i \in N(\bar{p}_{j_k})} \theta_i \geq \sum_{k=1}^{m_S} \mathcal{Q}_{\bar{p}_{j_k}} \geq \mathcal{L}_{S},
\end{align*}
where the last inequality follows from the definition of $\mathcal{L}_{S}$ and the fact that $\{\bar{p}_{j_k}\}_{k=1}^{m_S} \in \Pi(S,m_S)$. Therefore, the S-cuts are implied by the P-cuts, which have been shown to hold at $(x,\theta) = (x^{\nu},\theta^{\nu})$. 
\end{proof}

\rev{
\begin{remark}\label{rem:necessity}
Theorem~\ref{thm:validity_DL-shaped} establishes the superadditivity property as a necessary and sufficient condition for MP to be a valid reformulation of the VRPSD across all choices of the route-count restriction \(M \subseteq \{1,\ldots,n\}\). For a specific choice of $M$, the DL-shaped reformulation can remain valid even without superadditivity. This is the case for $M = \{n\}$, where the only feasible first-stage solution $(x^{\nu}, z^{\nu})$ consists of the singleton paths $\mathcal{P}(x^{\nu}) = \{(i)\}_{i \in N}$. Indeed, since $x^{\nu}_e = 0$ for every customer edge $e \in E(N)$, the resulting flow is $p(x^{\nu}) = 0$ for each path $p \in \mathcal{P}$. At $x = x^{\nu}$, the only non-trivial P-cuts \eqref{DL:p-cut} are thus those of the singleton paths, which impose $\theta_i \geq \mathcal{Q}_{(i)}$ for each $i \in N$. Regardless of whether the superadditivity property holds, the assignment $\theta^{\nu}_i = \mathcal{Q}_{(i)}$ then satisfies all P-cuts (and hence all S-cuts) and recovers the correct recourse cost: $\sum_{i \in N}\theta^{\nu}_i = \sum_{i \in N}\mathcal{Q}_{(i)} = \mathcal{Q}(x^{\nu})$. In contrast, the proof of Theorem~\ref{thm:validity_DL-shaped} establishes necessity from the case $M = \{1, \dots, n\}$, which is also the case relevant to the basic and ECC variants, making the superadditivity property necessary for the validity of the DL-shaped reformulation in both of these variants.
\end{remark}
}

\section{Properties of the recourse function}\label{sec:properties}
In Section \ref{sec:properties:monotonicity}, we review the key structural property of the DTD recourse function studied in \cite{parada2024disaggregated}, and evaluate its implications for the OR policy. We then show that the proof of the main result of \cite{parada2024disaggregated} relies on an incorrect argument. In Section \ref{sec:properties:superadditivity}, we provide a corrected proof of this result by leveraging the superadditivity property. Finally, we show that the superadditivity property is always satisfied under the OR policy.

\subsection{Monotonicity properties}\label{sec:properties:monotonicity}
\begin{definition}[Monotonicity property \cite{parada2024disaggregated}]\label{def:mono_prop_parada}
    An instance of the VRPSD satisfies the monotonicity property if, for any set $S \subseteq N$ with $|S| \geq 2$ such that $\sum_{i \in S}\mu_i \leq fQ$, for any pair $(a,b) \in S\times S$ with $a\neq b$, and any subset $\widetilde{S} \subseteq S \setminus \{a,b\}$, the following inequality holds for any positive integer $l \in \mathbb{N}$:
    \begin{equation}\label{ineq:mono_prop}
        \mathbb{P}\left[\sum_{i \in \widetilde{S} \cup \{a\} } \xi_i \leq lQ < \sum_{i \in \widetilde{S} \cup \{a,b\} } \xi_i \right] \geq \mathbb{P}\left[\sum_{i \in \widetilde{S} } \xi_i \leq lQ < \sum_{i \in \widetilde{S} \cup \{b\} } \xi_i \right]
    \end{equation}
\end{definition}

Under the monotonicity property, \cite{parada2024disaggregated} showed that the expected recourse cost of a feasible path is higher than or equal to that of any of its subsequences. \rev{This result, which we recall in Proposition \ref{prop:DTD_monotonic_subseq}, is true and is not affected by the incorrect argument used in \cite{parada2024disaggregated}}. However, we show in Proposition \ref{prop:OR_non_monotonic_subseq} that it does not extend to the OR policy.

\begin{proposition}[\cite{parada2024disaggregated}, Proposition 2] \label{prop:DTD_monotonic_subseq} Under the monotonicity property, any subsequence $p'$ of a path $p \in \mathcal{P}$ satisfies $\mathcal{Q}^{\text{DTD}}_{p'} \leq \mathcal{Q}^{\text{DTD}}_{p}$.
\end{proposition}

\begin{proposition}\label{prop:OR_non_monotonic_subseq}
Even if the monotonicity property holds, there may exist a subsequence $p'$ of a path $p \in \mathcal{P}$ such that $\mathcal{Q}^{\text{OR}}_{p'} > \mathcal{Q}^{\text{OR}}_{p}$.
\end{proposition}
\begin{proof} 
Consider the instance illustrated in Figure \ref{fig:non_monotonic_recourse}, with vehicle capacity $Q=20$, penalties $b^F=b^P=0$, and random demands $\xi_1,\xi_3 \sim \text{Poisson}(\lambda=9)$ and $\xi_2 \sim \text{Poisson}(\lambda=1)$. By Proposition 3 of \cite{parada2024disaggregated}, this instance satisfies the monotonicity property. Yet, we can verify that path $p=(1,2,3)$ and its subsequence $p'=(1,3)$ satisfy $\mathcal{Q}^{\text{OR}}_p \approx 3.25 < 6.08 \approx \mathcal{Q}^{\text{OR}}_{p'}$.
\end{proof}

\begin{figure}[H]
  \centering
  \begin{tikzpicture}[scale=1.705,
      every node/.style={font=\small},
      nd/.style ={circle,   draw, thick, minimum size=5mm, inner sep=0pt},
      rt/.style ={rectangle,draw, thick, minimum size=5mm, inner sep=0pt}]

    \node[rt] (N0) at ( 0,  0)   {0};
    \node[nd] (N1) at (-1.2, 1.8) {1};
    \node[nd] (N2) at ( 0,  0.8) {2};
    \node[nd] (N3) at ( 1.2, 1.8) {3};

    \draw (N0) to[bend left=15]  node[left]                  {$12$} (N1);
    \draw (N0) --                 node[left,  inner sep=0.8pt] {$2$} (N2);
    \draw (N0) to[bend right=15] node[right]                 {$12$} (N3);
    \draw (N2) to[bend right=15] node[right]                 {$10$} (N1);
    \draw (N2) to[bend left=15]  node[left]                  {$10$} (N3);
    \draw (N1) --                 node[above]                  {$8$} (N3);

  \end{tikzpicture}
  \caption{Instance with non-monotone OR recourse function \rev{(edges labeled with their cost $c_e$)}}
  \label{fig:non_monotonic_recourse}
\end{figure}

The main result of \cite{parada2024disaggregated}, given in their Proposition 10, states that the P-cuts \eqref{DL:p-cut} are valid optimality cuts under the DTD recourse policy if the monotonicity property holds. The proof of this proposition leverages Proposition~\ref{prop:DTD_monotonic_subseq} by claiming that, if a path $p \in \mathcal{P}(x^{\nu})$ satisfies $\mathcal{Q}^{\text{DTD}}_p \geq \mathcal{Q}^{\text{DTD}}_{p'}$ for any of its subsequences $p'$, then there exists a valid assignment $\{\theta_i^{\nu}\}_{i \in N(p)}$ of the variables $\{\theta_i\}_{i \in N(p)}$ such that $\sum_{i \in N(p)} \theta_i^{\nu} = \mathcal{Q}^{\text{DTD}}_p$. In Proposition \ref{prop:monotonicity_subsequence_not_sufficient}, we show this claim to be false. \rev{We exhibit a path $p=(p_1,p_2)$ whose expected recourse cost under the DTD policy exceeds that of any of its subsequences, while the P-cuts of paths $p_1$ and $p_2$ jointly impose $\sum_{i \in N(p)} \theta_i > \mathcal{Q}^{\text{DTD}}_{p}$ if $p \in \mathcal{P}(x^{\nu})$.}

\begin{proposition}\label{prop:monotonicity_subsequence_not_sufficient} A path $p \in \mathcal{P}$ may satisfy $\mathcal{Q}^{\text{DTD}}_p \geq \mathcal{Q}^{\text{DTD}}_{p'}$ for each of its subsequences $p'$, while $\mathcal{Q}^{\text{DTD}}_{p} < \mathcal{Q}^{\text{DTD}}_{p_1} + \mathcal{Q}^{\text{DTD}}_{p_2}$ for some $p_1, p_2 \in \mathcal{P}$ such that $p=(p_1,p_2)$.
\end{proposition}


\rev{
\begin{proof}
Assume all customers $i \in N=\{1,2,3,4,5,6,7,8\}$ have i.i.d.\ demands $\xi_i \sim \text{Bern}(\mu)$, and let the vehicle capacity be $Q=3$. In this setting, under the DTD policy, a failure occurs at the $j$-th customer served by a vehicle if and only if its demand is $\xi_j=1$, and the total demand of the first $j-1$ customers is a nonzero multiple of 3. For $j\leq 8$, we can thus write the expected number of failures at the $j$-th served customer as $\phi_j := \mu \cdot \mathbb{P}[\mathrm{Bin}(j-1,\mu) \in \{3,6\} ]$. This gives $\phi_1=\phi_2=\phi_3=0$, $\phi_4=\mu^4$, $\phi_5=4\mu^4(1-\mu)$, $\phi_6=10\mu^4(1-\mu)^2$, $\phi_7=20\mu^4(1-\mu)^3+\mu^7$, and $\phi_8=35\mu^4(1-\mu)^4+7\mu^7(1-\mu)$.

Now, assume $N$ is partitioned into $A=\{1,4,5,8\}$ and $B=\{2,3,6,7\}$, where every customer $i \in A$ is at distance $c_{0,i} = \frac{1}{2}$ from the depot and every customer $i \in B$ is located at the depot, i.e. $c_{0,i} = 0$. Take $b^F=0$, so that the failure cost at customer $i \in N$ is $c^F_i = 1$ if $i \in A$ and $c^F_i = 0$ otherwise. Denoting by $\boldsymbol{1}_{A}$ the indicator function of set $A$, i.e., $\boldsymbol{1}_{A}(i)=1$ if $i \in A$ and $\boldsymbol{1}_{A}(i)=0$ otherwise, we can then write the expected DTD recourse of any path $p'=(i_1,\dots,i_t)$ as $\bar{\mathcal{Q}}_{p'}^{\text{DTD}} = \sum_{j=1}^t \phi_j \boldsymbol{1}_{A}(i_j)$.

Take $\mu=\frac{9}{10}$ and a load factor $f \geq 8\mu/Q$ so that $p=(1,2,3,4,5,6,7,8)$ is feasible, i.e., $p \in \mathcal{P}$, and take $p_1=(1,2,3,4)$ and $p_2=(5,6,7,8)$. In either orientation, paths $p_1$ and $p_2$ visit customers of $A$ in positions $1$ and $4$, whereas path $p$ visits customers of $A$ in positions $1$, $4$, $5$, and $8$. This gives:
\begin{align*}
    \mathcal{Q}^{\mathrm{DTD}}_{p} = \phi_4+\phi_5+\phi_8 = \frac{62782209}{50000000} < \frac{6561}{5000} = 2\phi_4 = \mathcal{Q}^{\mathrm{DTD}}_{p_1} + \mathcal{Q}^{\mathrm{DTD}}_{p_2}.
\end{align*}

It remains to show that any subsequence $p'$ of $p$ satisfies $\mathcal{Q}^{\mathrm{DTD}}_{p'}\leq \mathcal{Q}^{\mathrm{DTD}}_p$. Since this holds with equality for $p'=p$, assume that $p'$ is a proper subsequence of $p$. For $\mu=\frac{9}{10}$, we directly verify that $\phi_4 > \phi_7 > \phi_8 > \phi_5 > \phi_6$, and $\phi_5 + \phi_8 > \phi_7$; in particular, $\mathcal{Q}^{\mathrm{DTD}}_p > \phi_4+\phi_7$. If $|N(p')\cap A|\leq 2$, then $\mathcal{Q}^{\mathrm{DTD}}_{p'}\leq \phi_4+\phi_7<\mathcal{Q}^{\mathrm{DTD}}_p$, since $\phi_4$ and $\phi_7$ are the two largest values among $\{\phi_j\}_{j=1}^8$. Now suppose $|N(p')\cap A|\geq 3$. We consider three cases and show that $\mathcal{Q}^{\mathrm{DTD}}_{p'} \leq \phi_4 + \phi_7$ holds in each of them. (i) If $|N(p')\cap A|=3$, then $p'$ is a subsequence of $(1,2,3,4,5,6,7,8)$ that visits three out of the four customers of $A=\{1,4,5,8\}$. Therefore, $p'$ has an orientation in which a customer of $A$ occupies position 1. The two remaining customers of $N(p')\cap A$ contribute at most $\phi_4+\phi_7$ to the expected recourse cost, so $\mathcal{Q}^{\text{DTD}}_{p'} \leq \phi_1 + \phi_4 + \phi_7 = \phi_4 + \phi_7$. (ii) If $A \subseteq N(p')$ and $|N(p') \cap \{2,3\}| \leq 1$, then visiting $p'$ in its canonical orientation places customer 1 in position 1 and customer 4 in position 2 or 3. Since $\phi_1=\phi_2=\phi_3=0$, the probability of failure at these two customers of $A$ is zero. As before, two customers of $A$ remain, hence $\mathcal{Q}^{\text{DTD}}_{p'} \leq \phi_4 + \phi_7$. (iii) If $A \subseteq N(p')$ and $|N(p') \cap \{6,7\}| \leq 1$, then visiting $p'$ in its reverse orientation places customer 8 in position 1 and customer 5 in position 2 or 3, both with zero failure probability, so again $\mathcal{Q}^{\text{DTD}}_{p'} \leq \phi_4 + \phi_7$. This exhausts the relevant cases, as otherwise either $|N(p')\cap A|\leq 2$ or $p'=p$. 
\end{proof}
}

\rev{By invalidating the proof of Proposition 10 of \cite{parada2024disaggregated}, Proposition~\ref{prop:monotonicity_subsequence_not_sufficient} raises the question of whether the monotonicity property is a sufficient condition for the validity of the P-cuts~\eqref{DL:p-cut} under the DTD policy. We answer this question affirmatively in the next section, by establishing the validity via the superadditivity property rather than via subsequence monotonicity.}

\subsection{Superadditivity properties}\label{sec:properties:superadditivity}
In this section, we apply Theorem \ref{thm:validity_DL-shaped} to verify the validity of the DL-shaped reformulation under the DTD and OR policies. In Theorem \ref{thm:Q_DTD_superadditive}, we show that the monotonicity property implies the superadditivity property for the DTD policy, hence that Proposition 10 of \cite{parada2024disaggregated} holds despite the incorrect argument of its original proof. We then show in Theorem \ref{thm:Q_superadditive} that the superadditivity property always holds under the OR policy.

\begin{theorem}\label{thm:Q_DTD_superadditive} 
    The superadditivity property holds for $\mathcal{Q}=\mathcal{Q}^{\text{DTD}}$ if the monotonicity property is satisfied.
\end{theorem}
\begin{proof}
Let $p_1,p_2 \in \mathcal{P}$ be two paths such that $(p_1,p_2)\in \mathcal{P}$. Let $p = (p_1,p_2) = (i_1,\dots,i_t)$ be their concatenation. Without loss of generality, we assume that $\bar{\mathcal{Q}}^{\text{DTD}}_p = \mathcal{Q}^{\text{DTD}}_p$, i.e., $p$ is in its best orientation. Write $p_1=(i_1,\dots,i_s)$ and $p_2=(i_{s+1},\dots,i_t)$. Starting from equation \eqref{rec:DTD}, the proof is direct:
\begingroup
\allowdisplaybreaks
\begin{align*}
    \mathcal{Q}^{\text{DTD}}_p &= \sum_{j=1}^t\sum_{l=1}^{+\infty}\mathbb{P}\left[ \sum_{k=1}^{j-1}\xi_{i_k} \leq lQ < \sum_{k=1}^j\xi_{i_k} \right]c^F_{i_{j}}\\
    &= \bar{\mathcal{Q}}^{\text{DTD}}_{p_1} + \sum_{j=s+1}^t\sum_{l=1}^{+\infty} \mathbb{P}\left[ \sum_{k=1}^{j-1}\xi_{i_k} \leq lQ < \sum_{k=1}^j\xi_{i_k} \right]c^F_{i_{j}}\\
    &\geq \bar{\mathcal{Q}}^{\text{DTD}}_{p_1} + \sum_{j=s+1}^t\sum_{l=1}^{+\infty}\mathbb{P}\left[ \sum_{k=s+1}^{j-1}\xi_{i_k} \leq lQ < \sum_{k=s+1}^j\xi_{i_k} \right]c^F_{i_{j}}\\
    &= \bar{\mathcal{Q}}^{\text{DTD}}_{p_1} + \bar{\mathcal{Q}}^{\text{DTD}}_{p_2}\\
    &\geq \mathcal{Q}^{\text{DTD}}_{p_1} + \mathcal{Q}^{\text{DTD}}_{p_2}.
\end{align*}
\endgroup
The first inequality follows from $s$ successive applications of \eqref{ineq:mono_prop} to each element of the summation.
\end{proof}

Before moving to Theorem \ref{thm:Q_superadditive}, we show that the cost-to-go function of the OR policy, \rev{as defined in Section \ref{sec:OR_pol}}, is monotone. 
\begin{lemma}
\label{lemma:monotonicity_cost_to_go}
For any \rev{route $(0, p, 0) = (i_0, i_{1}, \dots, i_{t}, i_{t+1})$ and every position $j \in \{0,\dots,t\}$}, the expected recourse cost-to-go $F^{p}_{i_{j}}(q)$ is non-increasing in the residual capacity \rev{$q \in \{0,\dots,Q\}$}.
\end{lemma}
\begin{proof}
\rev{The proof is by reverse induction on $j\in\{0,\dots,t\}$. We first prove the claim directly for $j=t$ and $j=t-1$, and then show that, for each $j\in\{1,\dots,t-1\}$, monotonicity of $F^p_{i_j}(\cdot)$ implies monotonicity of
$F^p_{i_{j-1}}(\cdot)$.

The claim holds trivially for
$j=t$, since $F^p_{i_t}(q)=0$ for all $q \in \{0,\dots,Q\}$. We next verify the claim for
$j=t-1$. Since $F^p_{i_{t-1}}(q)
=
\min\{H^p_{i_t}(q),H^{*p}_{i_{t-1},i_t}\}$,
and the preventive-restocking term $H^{*p}_{i_{t-1},i_t}$ is independent of $q$, it is enough to
show that $H^p_{i_t}(\cdot)$ is non-increasing. Take integers $q_1,q_2$ with $0\le q_1<q_2\le Q$. Using that $F^p_{i_t}(q)=0$ for all $q \in \{0,\dots,Q\}$, we can write:
\[
H^p_{i_t}(q_1)-H^p_{i_t}(q_2)
=
\sum_{s\in\Xi_{i_t}}
c^F_{i_t}\bigl(\Psi(s,q_1)-\Psi(s,q_2)\bigr)\rho_{i_t}^s,
\]
which is non-negative since $\Psi(s,q)=\max\{\lceil(s - q)/Q\rceil,0\}$ is non-increasing in $q$ for any given $s$. Hence
$F^p_{i_{t-1}}(\cdot)$ is non-increasing. 

Now fix $j\in\{1,\dots,t-1\}$ and suppose that $F^p_{i_j}(\cdot)$ is non-increasing. We show that $F^p_{i_{j-1}}(\cdot)$ is non-increasing. Again, as $F^p_{i_{j-1}}(q) = \min\{H^p_{i_j}(q),H^{*p}_{i_{j-1},i_j}\}$,
and the preventive-restocking term is independent of $q$, it is enough to show that $H^p_{i_j}(\cdot)$ is non-increasing.

Take integers $q_1,q_2$ with $0\le q_1<q_2\le Q$. We show that $H^p_{i_j}(q_1)\ge H^p_{i_j}(q_2)$. For $\ell\in\{1,2\}$, denote by $\psi_\ell(s):=\Psi(s,q_\ell)$ and $r_\ell(s):=\psi_\ell(s)Q+q_\ell-s$ the number of restocking trips required to satisfy demand $s$ and the resulting residual capacity after serving customer $i_j$, respectively, when the vehicle arrives at $i_j$ with residual capacity $q_\ell$. Since $0\le q_1<q_2\le Q$, we have $\psi_1(s)-\psi_2(s)\in\{0,1\}$ for all $s\ge 0$. Thus, expanding the definition of $H^p_{i_j}$ and splitting the demand realizations according to whether starting from $q_2$ instead of $q_1$ reduces the number of required restocking trips, we obtain:
\begin{align}
H^p_{i_j}(q_1)-H^p_{i_j}(q_2)
=&
\sum_{s\in\Xi_{i_j}}
\left[
c^F_{i_j}(\psi_1(s)-\psi_2(s))
+
F^p_{i_j}(r_1(s))-F^p_{i_j}(r_2(s))
\right]\rho_{i_j}^s
\notag\\
=&
\sum_{\substack{s\in\Xi_{i_j}\\ \psi_1(s)=\psi_2(s)}}
\left[
F^p_{i_j}(r_1(s))-F^p_{i_j}(r_2(s))
\right]\rho_{i_j}^s
\notag\\
&+
\sum_{\substack{s\in\Xi_{i_j}\\ \psi_1(s)=\psi_2(s)+1}}
\left[
c^F_{i_j}
+
F^p_{i_j}(r_1(s))-F^p_{i_j}(r_2(s))
\right]\rho_{i_j}^s .
\label{eq:split}
\end{align}
The first sum in \eqref{eq:split} is non-negative by the induction hypothesis: when
$\psi_1(s)=\psi_2(s)$, we have $r_1(s)<r_2(s)$, and
$F^p_{i_j}(\cdot)$ is non-increasing.

It remains to bound the terms in the second sum. If
$\psi_1(s)=\psi_2(s)+1$, then $r_1(s)\in\{0,\dots,Q\}$, so the induction hypothesis
gives $F^p_{i_j}(r_1(s))\ge F^p_{i_j}(Q)$.
Moreover, by the Bellman recursion \eqref{def:cost-to-go} at node $i_j$, $F^p_{i_j}(r_2(s))
\le
H^{*p}_{i_j,i_{j+1}}
=
c^P_{i_j,i_{j+1}}+H^p_{i_{j+1}}(Q)$.
Therefore, each term in the second summation satisfies:
\begin{align*}
    c^F_{i_j} + F^p_{i_j}(r_1(s))-F^p_{i_j}(r_2(s)) & \geq c^F_{i_j}
+
F^p_{i_j}(Q)
-
c^P_{i_j,i_{j+1}}
-
H^p_{i_{j+1}}(Q) \\
&= c^F_{i_j}-c^P_{i_j,i_{j+1}} \\
&= \left(c_{0,i_j}+c_{i_j,i_{j+1}}-c_{0,i_{j+1}}\right)
+
\left(b^F-b^P\right) \\
& \geq 0.
\end{align*}
The first equality uses that $F^p_{i_j}(Q)=\min\{H^p_{i_{j+1}}(Q),c^P_{i_j, i_{j+1}} + H^p_{i_{j+1}}(Q)\}=H^p_{i_{j+1}}(Q)$, i.e., proceeding directly is always optimal at full capacity. The second equality uses the definitions of $c^F_{i_j}$ and $c^P_{i_j,i_{j+1}}$. The last inequality follows from the triangle inequality and
$b^F\ge b^P$. 

Hence both sums in \eqref{eq:split} are non-negative, and
$H^p_{i_j}(q_1)\ge H^p_{i_j}(q_2)$. We conclude that $H_{i_j}^p(\cdot)$ is non-increasing, and hence so is $F^p_{i_{j-1}}(\cdot)$.
}
\end{proof}

\begin{theorem}\label{thm:Q_superadditive}
    The superadditivity property holds for $\mathcal{Q}=\mathcal{Q}^{\text{OR}}$.
\end{theorem}
\begin{proof}
Let $p_1,p_2\in\mathcal{P}$ such that $(p_1,p_2)\in\mathcal{P}$. Let $p = (p_1,p_2) = (i_1,\dots,i_t)$ be their concatenation. Without loss of generality, we assume that $\bar{\mathcal{Q}}^{\text{OR}}_p = \mathcal{Q}^{\text{OR}}_p$, i.e., $p$ is in its best orientation. Write $p_1=(i_1,\dots,i_j)$ and $p_2=(i_{j+1},\dots,i_t)$. Let $\xi=(\xi_1,\dots,\xi_t)$ denote the vector of random demands along route $(0,p,0)$. Apply the OR policy on that route, and define the following two random cost components:
\begin{enumerate}
    \item $C_1(\xi)$ := total recourse cost incurred \emph{up to and including}
            the service of customer $i_j$.
    \item $C_2(\xi)$ := total recourse cost incurred \emph{after}
            the service of customer $i_j$.
\end{enumerate}

Clearly, $C_1(\xi)+C_2(\xi)$ is the total recourse cost on $p$, hence the expected recourse cost satisfies:
\begin{equation}\label{eq:superadd_Q_decomp}
    \mathcal{Q}^{\text{OR}}_{p} = \mathbb E\left[C_1(\xi)+C_2(\xi)\right].
\end{equation}

\paragraph{Prefix bound.}
First, by applying the OR decisions designed for $(0,p,0)$ until customer $i_j$ is served, and then returning to the depot as planned in the a priori route $(0,p_1,0)$, we obtain a suboptimal recourse policy for route $(0,p_1,0)$, with expected cost $\mathbb E[C_1(\xi)]$. Because $\bar{\mathcal{Q}}^{\text{OR}}_{p_1}$ is, by construction of the OR policy, the minimum expected recourse cost achievable on that route, we obtain:
\begin{equation}\label{eq:superadd_prefix}
    \mathbb E[C_1(\xi)] \geq \bar{\mathcal{Q}}^{\text{OR}}_{p_1} \geq \mathcal{Q}^{\text{OR}}_{p_1}.
\end{equation}

\paragraph{Suffix bound.}
Condition on the demands $\xi_{i_1:i_j}:=(\xi_{i_1},\dots,\xi_{i_j})$.  
Given these values, the residual capacity after serving $i_j$ is uniquely determined as $q(\xi_{i_1:i_j})\in \rev{\{0,\dots,Q\}}$. From there, expanding the conditional expectation of $C_2(\xi)$ using the Bellman recursion \eqref{def:cost-to-go} yields:
\begingroup
\allowdisplaybreaks
\begin{align*}
    \mathbb E\bigl[C_2(\xi)
                 \,\bigl\lvert\,
                 \xi_{i_1:i_j}\bigr] &= F^{p}_{i_j}\bigl(q(\xi_{i_1:i_j})\bigr)\\
    &= \min\left\{ H^p_{i_{j+1}}\bigl(q(\xi_{i_1:i_j})\bigr), \ H^{*p}_{i_j,i_{j+1}} \right\} \\
    &= \min\left\{ H^{p_2}_{i_{j+1}}\bigl(q(\xi_{i_1:i_j})\bigr), \  c^P_{i_j, i_{j+1}}+ H^{p_2}_{i_{j+1}}(Q) \right\} \\
    &\geq H^{p_2}_{i_{j+1}}(Q)\\
    &= \bar{\mathcal{Q}}^{\text{OR}}_{p_2}\\
    &\geq \mathcal{Q}^{\text{OR}}_{p_2}.
\end{align*}
\endgroup

The first inequality follows from the monotonicity of $H^{p_2}_{i_{j+1}}(\cdot)$, demonstrated in Lemma~\ref{lemma:monotonicity_cost_to_go}, and the non-negativity of the preventive recourse action cost $c^{P}_{i_j,i_{j+1}}$.
Because the bound holds for every realization of $\xi_{i_1:i_j}$, the \rev{law of total expectation} gives:
\begin{equation}\label{eq:superadd_suffix}
  \mathbb E[C_2(\xi)] \geq \mathcal{Q}^{\text{OR}}_{p_2}.
\end{equation}

By adding~\eqref{eq:superadd_prefix} and~\eqref{eq:superadd_suffix} and substituting the resulting inequality into~\eqref{eq:superadd_Q_decomp}, we obtain that $\mathcal{Q}^{\text{OR}}_{p} \geq \mathcal{Q}^{\text{OR}}_{p_1} + \mathcal{Q}^{\text{OR}}_{p_2}$.
\end{proof}



\section{New valid inequalities} \label{sec:E-cuts}

The performance of the DL-shaped method critically relies on the strength of its LBFs. However, the S-cuts \eqref{DL:s-cut} cannot efficiently tighten the approximation of the recourse function for sets $S \subseteq N$ that can be partitioned into $m_S$ paths with \rev{near-zero} recourse costs. This is a common \rev{occurrence} under the OR policy, where recourse costs can often be moved to the first stage by constructing long routes that make available cheap preventive recourse actions. To obtain stronger LBFs, we propose a generalization of the S-cuts that bounds the recourse for a restricted set of the feasible partitions of $S$ into paths. This subset is defined based on the edges that can compose these paths, and we thus call our new valid inequalities the edge-set cuts (E-cuts).

\rev{Section~\ref{sec:E-cuts:def_and_validity} defines the E-cuts and shows they generalize the P-cuts and the S-cuts. Section~\ref{sec:ecut-prs} analytically compares the E-cuts with the partial route-split inequalities of~\cite{hoogendoorn2023improved}. The remainder of the section then focuses on the OR policy, which is the setting in which we implement the DL-shaped method computationally. Section~\ref{sec:E-cuts:wheel} exhibits a family of instances where a single E-cut suffices for convergence of the DL-shaped method, whereas at least $n$ valid inequalities from the literature would be necessary. Section~\ref{sec:E-cuts:strategy} describes our edge-set selection strategy, and Section~\ref{sec:bounds} presents the lower bounds we use to instantiate the E-cuts under the OR policy.}


\subsection{Definition and basic properties of the E-cuts} \label{sec:E-cuts:def_and_validity}

For a set of customers $S \subseteq N$ and an edge set $E_S \subseteq E(S)$, let $\bar{\mathcal{P}}(E_S) \subseteq \mathcal{P}$ denote the set of feasible paths on the subgraph $(S,E_S)$ of $G$. Define $\bar{\Pi}(E_S,m) := \left\{ \pi \in \Pi(S,m) : \pi \subseteq \bar{\mathcal{P}}(E_S) \right\}$ as the set of all partitions of $S$ into exactly $m$ feasible paths formed only of edges of $E_S$, $\bar{m}(E_S) := \min\{m \in \mathbb{N} : \bar{\Pi}(E_S,m) \neq \emptyset\}$ as the minimum number of vehicles needed to cover the customers of $S$ on graph $(S,E_S)$, and $\bar{\mathcal{L}}(E_S, m) := \min_{ \pi \in \bar{\Pi}(E_S,m)} \sum_{p \in \pi} \mathcal{Q}_{p}$ as the smallest recourse cost achievable by partitioning $S$ into exactly $m$ of these feasible paths. The E-cut of the edge set $E_S$ is given by:
\begin{equation}
    \sum_{i \in S} \theta_{i} \geq \mathcal{L}_{E_S}\left( \sum_{e \in E_S}x_{e} - |S| + m_{E_S} + 1 \right), \label{DL:e-cut}
\end{equation}
where $1 \leq m_{E_S} \leq \bar{m}(E_S)$ and $0 \leq \mathcal{L}_{E_S}\leq \bar{\mathcal{L}}(E_S, m_{E_S})$ are fixed values.

\begin{theorem}\label{thm:validity_E-cuts}
The E-cut \eqref{DL:e-cut} is valid if the superadditivity property holds.
\end{theorem}
\begin{proof}
\rev{Let $(x^{\nu},z^{\nu},\theta^{\nu})$ be a feasible solution of the DL-shaped formulation MP, which, by Theorem \ref{thm:validity_DL-shaped}, is a valid reformulation of \eqref{objective}--\eqref{z_val} under superadditivity. Fix $S \subseteq N$, $E_S \subseteq E(S)$, $1 \leq m_{E_S} \leq \bar{m}(E_S)$, and $0 \leq \mathcal{L}_{E_S}\leq \bar{\mathcal{L}}(E_S, m_{E_S})$. We show that the resulting E-cut is satisfied at $(x,\theta) = (x^{\nu},\theta^{\nu})$.
The factor multiplying $\mathcal{L}_{E_S}$ in the E-cut can be non-positive at $x=x^\nu$, in which case the cut is trivially respected. Otherwise, by feasibility of $(x^\nu,z^\nu)$ and the definition of $\bar m(E_S)$, $\sum_{e \in E_S} x^\nu_e \leq |S|-\bar m(E_S) \leq |S|-m_{E_S}$, so the factor is at most one; being a strictly positive integer, it is exactly one. Equivalently, the customers of $S$ are covered by exactly $m_{E_S}$ subpaths of $\mathcal{P}(x^\nu)$, each formed only of edges in $E_S$. Denote these subpaths by $\{\bar{p}_{j_k}\}_{k=1}^{m_{E_S}} \in \bar{\Pi}(E_S,m_{E_S})$, where for each $k \in \{1,\dots,m_{E_S}\}$, $j_k \in \{1,\dots,m\}$ is the index of the path of $\mathcal{P}(x^\nu)$ of which $\bar{p}_{j_k}$ is a subpath.
By construction, $\bar{p}_{j_k}(x^\nu)=|N(\bar{p}_{j_k})|-1$ for each of these subpaths. Since $(x^\nu,z^\nu,\theta^\nu)$ is feasible for MP, the P-cuts \eqref{DL:p-cut} of these subpaths are satisfied, and summing them at $x=x^\nu$ gives:
\begin{align*}
    \sum_{i \in S}\theta^{\nu}_i
    = \sum_{k=1}^{m_{E_S}}\sum_{i \in N(\bar{p}_{j_k})} \theta^{\nu}_i
    \geq \sum_{k=1}^{m_{E_S}} \mathcal{Q}_{\bar{p}_{j_k}}
    \geq \bar{\mathcal{L}}(E_S, m_{E_S})
    \geq \mathcal{L}_{E_S},
\end{align*}
where the second inequality follows from the definition of $\bar{\mathcal{L}}(E_S, m_{E_S})$, since $\{\bar{p}_{j_k}\}_{k=1}^{m_{E_S}} \in \bar{\Pi}(E_S,m_{E_S})$. The E-cut is thus respected at $(x,\theta)=(x^\nu,\theta^\nu)$.}

\end{proof}

\begin{proposition}\label{prop:E-cuts-generalize}
The E-cuts \eqref{DL:e-cut} generalize the P-cuts \eqref{DL:p-cut} and the S-cuts \eqref{DL:s-cut}.
\end{proposition}

\begin{proof}
The S-cut associated with a set $S \subseteq N$ is obtained by taking $E_S = E(S)$. The P-cut of a path $p = (i_1,\dots,i_t) \in \mathcal{P}$ corresponds to the strongest valid E-cut for customers $S=N(p)$ and the edge set $E_S = \{\{i_j,i_{j+1}\}\}_{j=1}^{t-1}$.
\end{proof}

\rev{We summarize in Table~\ref{Tab:cuts_summary} the structure of the three families of DL-shaped cuts. Each is indexed by a different object: a path, a customer set, or an edge set. For a given set of customers $S$, the P-cuts and S-cuts arise as the two extreme cases of the E-cut family, corresponding respectively to the smallest valid edge set (the $|S|-1$ edges of a single path) and the largest ($E(S)$, the set of all edges with both endpoints in $S$). By Theorems~\ref{thm:validity_DL-shaped} and~\ref{thm:validity_E-cuts}, all three are valid under the superadditivity property.}

\begin{table}[H]
\centering
\rev{
\caption{Structure of the P-cuts, S-cuts, and E-cuts} \label{Tab:cuts_summary}
\begin{tabular*}{0.8\textwidth}{lccc}
\toprule
Cut family & Indexed by & Active at an integer solution $x$ if & Coefficient \\
\midrule
P-cut \eqref{DL:p-cut} & path $p \in \mathcal{P}$ & $p$ is a subpath of $\mathcal{P}(x)$ & $\mathcal{Q}_p$ \\
\addlinespace
\addlinespace
S-cut \eqref{DL:s-cut} & set $S \subseteq N$ & $m_S$ subpaths of $\mathcal{P}(x)$ partition $S$ & $\mathcal{L}_S$ \\
\addlinespace
\multirow{2}{*}{E-cut \eqref{DL:e-cut}} & \multirow{2}{*}{edge set $E_S \subseteq E(S)$} & $m_{E_S}$ subpaths of $\mathcal{P}(x)$ cover $S$ & \multirow{2}{*}{$\mathcal{L}_{E_S}$} \\
 & & using only edges in $E_S$ & \\
\bottomrule
\end{tabular*}
}
\end{table}


\rev{
\subsection{Relation to the partial route-split inequalities} \label{sec:ecut-prs}

We recall the framework of~\cite{hoogendoorn2023improved}. A \emph{partial route} \cite[Section~4.1]{hoogendoorn2023improved} is an ordered sequence $h = (U_0, U_1, \ldots, U_b)$ of subsets of~$N_0$, for $b \geq 2$, that (i) starts and ends at the depot, i.e., $U_0 = U_b = \{0\}$, and (ii) visits disjoint subsets of customers $U_{j} \subseteq N$, $j \in \{1,\dots,b-1\}$, where (iii) at least one of any two consecutive subsets is a singleton, i.e., $|U_j|\geq 2 \implies |U_{j+1}|=1$, $j \in \{1,\dots,b-1\}$. We denote by $N(h) := \bigcup_{j=1}^{b-1} U_j$ the set of customers covered by $h$. A subset~$U_j$ with $|U_j| \geq 2$ is an \emph{unstructured component}: its customers may be visited in any order. If $h$ does not comprise unstructured components, it determines a unique route, and we call it \emph{fully structured}. A route $(0,p,0)$ with $N(p)=N(h)$ is said to \emph{adhere} to~$h$ if it respects the ordering prescribed by~$h$; we denote by $R_h$ the set of all routes that adhere to~$h$ in at least one of their two orientations. A \emph{partial route activation function} $W_h(x)$ equals~$1$ at every integer feasible solution containing a route in $R_h$, and is nonpositive at all other integer feasible solutions \cite[Section~4.3]{hoogendoorn2023improved}. Given a partial route $h$ and a lower bound~$P_h \leq \mathcal{Q}_p$ on the expected recourse cost of every route $(0,p,0) \in R_h$, the \emph{partial route-split inequality} \cite[equation~(16)]{hoogendoorn2023improved} reads:
\begin{equation}\label{eq:prs}
  \theta_{v(h)} \geq P_h \, W_h(x),
\end{equation}
where $v(h)$ denotes the lowest-indexed customer in~$N(h)$.

In addition to partial route-split inequalities, \cite{hoogendoorn2023improved} introduce \textit{route-split inequalities}, which are the special case where $h$ is fully structured, and \textit{multi-route-split inequalities}, which bound the aggregate
recourse of a customer set~$S$ without prescribing a partial ordering. The partial route-split inequalities occupy an intermediate position between the route-split inequalities, which are tight but active only at solutions containing a specific route, and the multi-route-split inequalities, which are active for a wider set of solutions but provide weaker bounds. E-cuts similarly interpolate between P-cuts and S-cuts. However, the two families achieve this flexibility through different mechanisms: partial route-split inequalities vary the partial ordering of customers, whereas E-cuts vary the set of edges available to form routes. In Theorem~\ref{thm:ecut-prs}, we establish that under superadditivity, for every partial route-split inequality \eqref{eq:prs}, there exists an E-cut \eqref{DL:e-cut} providing at least as strong a bound on the aggregate recourse $\sum_{i \in N(h)}\theta_i$ at every integer feasible solution.

\begin{theorem}\label{thm:ecut-prs}
Let $h = (U_0, U_1, \ldots, U_b)$ be a partial route and $P_h \geq 0$ a lower bound on $\mathcal{Q}_p$ for any route $(0,p,0) \in R_h$. If the superadditivity property holds, the E-cut~\eqref{DL:e-cut} with $S = N(h)$, $E_S = E_S^h :=  \bigcup_{j=1}^{b-1} E(U_j) \cup \bigcup_{j=1}^{b-2} \delta(U_j, U_{j+1})$, $m_{E_S}=1$ and $\mathcal{L}_{E_S} = P_h$ is valid. Moreover, at every integer feasible solution~$x$, the following inequality holds and can be strict:
\begin{equation}\label{eq:ecut-vs-prs}
  \min\Bigl\{\sum_{i \in N(h)}\theta_i :
  \text{\eqref{DL:e-cut}},\; \theta \geq 0\Bigr\}
  \geq
  \min\Bigl\{\sum_{i \in N(h)}\theta_i :
  \text{\eqref{eq:prs}},\; \theta \geq 0\Bigr\}.
\end{equation}
\end{theorem}

\begin{proof}
Fix any path $p$ such that $\{p\} \in \bar{\Pi}(E_S, 1)$; by definition, $p$ visits all customers of $S = N(h)$ using only edges of~$E_S$. For any unstructured component $U_j$, condition~(iii) in the definition of partial route ensures that $U_{j-1} = \{u\}$ and $U_{j+1} = \{w\}$ are singletons. Since $E_S$ contains no edges between non-consecutive components by construction, the only nodes outside $U_j$ adjacent to a node of $U_j$ in the graph $(S, E_S)$ are $u$ and $w$. It follows that $p$ traverses each unstructured component as a contiguous block, and hence $U_1, \ldots, U_{b-1}$ in sequence (in one of its two orientations). Therefore, $(0,p,0) \in R_h$, and so $P_h \leq \mathcal{Q}_p$. This holds for all $\{p\} \in \bar{\Pi}(E_S, 1)$, hence $P_h \leq \bar{\mathcal{L}}(E_S, 1)$. The E-cut is therefore valid under superadditivity by Theorem~\ref{thm:validity_E-cuts}.

Now fix an integer feasible solution~$x$. The left-hand side of~\eqref{eq:ecut-vs-prs} evaluates to $\max\{0,\; P_h(\sum_{e \in E_S} x_e - |S| + 2)\}$, and the right-hand side to $\max\{0,\; P_h\, W_h(x)\}$. If $W_h(x) = 1$, the solution contains a route $(0,p,0) \in R_h$, whose $|S|-1$ customer edges all lie in~$E_S^h$ by construction, so both sides equal~$P_h$. If $W_h(x) \leq 0$, the right-hand side equals~$0$, while the left-hand side is non-negative.

The inequality \eqref{eq:ecut-vs-prs} can be strict. Take a partial route $h=(\{0\},\{1\},\{0\})$ with $P_h > 0$. The E-cut \eqref{DL:e-cut} associated with $h$ is obtained by taking $S=\{1\}$, $E_S= \emptyset$, $m_{E_S}=1$ and $\mathcal{L}_{E_S} = P_h$. It reads $\theta_1 \geq P_h(\sum_{e \in \emptyset} x_e - 1 + 2) = P_h$ and is thus active at every solution, whereas the partial route-split inequality \eqref{eq:prs} is only active at an integer feasible solution if it contains route $(0,1,0)$.
\end{proof}

\begin{remark}\label{rem:dominance}
The result established in Theorem~\ref{thm:ecut-prs} concerns the aggregate recourse bound $\sum_{i \in N(h)}\theta_i$ at integer feasible solutions, and is not a dominance relation in the polyhedral sense. The partial route-split inequality~\eqref{eq:prs} additionally constrains the individual variable~$\theta_{v(h)}$, which may provide complementary tightening in the continuous relaxation. We note that E-cuts cannot be adapted to constrain a single variable~$\theta_v$ in a similar fashion. Indeed, unlike partial route-based cuts, E-cuts for nested edge sets $E_{S'} \subset E_S$ can be simultaneously active at the same integer feasible solution, and assigning their bounds to individual variables would lead to compounding. For instance, consider the P-cuts for $p = (2)$ and $p' = (1,2)$ with $\mathcal{Q}_p = 1$ and $\mathcal{Q}_{p'} = 2$. At the solution corresponding to route $(0,1,2,0)$, both P-cuts are active. If they were indexed as $\theta_2 \geq 1$ and $\theta_1 \geq 2$, their sum would force $\theta_1 + \theta_2 \geq 3 > 2 = \mathcal{Q}_{p'}$, exceeding the true recourse cost. The $\sum_{i \in S}\theta_i$ formulation avoids this: $\theta_2 \geq 1$ and $\theta_1 + \theta_2 \geq 2$ are jointly valid.
\end{remark}


\subsection{One E-cut can replace $n$ cuts from the literature} \label{sec:E-cuts:wheel}

In this section, we show that the E-cuts provide a fundamental advantage over existing valid inequalities for approximating the OR recourse function. Specifically, we construct a family of instances for which the DL-shaped method can terminate after generating a single E-cut, while at least~$n$ cuts from~\cite{hoogendoorn2023improved} and~\cite{parada2024disaggregated} are needed, a number that grows arbitrarily with the instance size.

\paragraph{Instance family.} Consider the wheel graph formed by connecting $n\geq4$~customers by the cycle $E_1 := \{ e_k \}_{k=1}^n$, where $e_k := \{k, k+1\}$ for $k \leq n-1$ and $e_{n} := \{n,1\}$, and connecting the depot~$0$ to every customer $i \in N$. Define the traveling cost between nodes as their distance on this graph, i.e., $c_{0,i}=1$ for $i\in N$, and for distinct customers $i,j \in N$, $c_{i,j}=1$ if $\{i,j\}\in E_1$ and $c_{i,j}=2$ otherwise. As distances on a graph, these costs satisfy the triangle inequality. We refer to the edges in~$E_1$ as \emph{boundary edges} and to those in $E_2 := E(N) \setminus E_1$ as \emph{diagonal edges}. The resulting instance family is illustrated in Figure~\ref{fig:wheel_graphs}.

\begin{figure}[htbp]\rev{
  \centering
  \begin{tikzpicture}[scale=1.705, baseline=(A0.center),
      every node/.style={font=\small},
      nd/.style ={circle,   draw, thick, minimum size=5mm, inner sep=0pt},
      rt/.style ={rectangle,draw, thick, minimum size=5mm, inner sep=0pt}]
    \node[rt] (A0) at ( 0, 0) {0};
    \node[nd] (A1) at ( 0, 1) {1};
    \node[nd] (A2) at ( 1, 0) {2};
    \node[nd] (A3) at ( 0,-1) {3};
    \node[nd] (A4) at (-1, 0) {4};
    \draw (A0) -- (A1);
    \draw (A0) -- (A2);
    \draw (A0) -- (A3);
    \draw (A0) -- (A4);
    \draw[very thick] (A1) -- (A2);
    \draw[very thick] (A2) -- (A3);
    \draw[very thick] (A3) -- (A4);
    \draw[very thick] (A4) -- (A1);
    \draw[dotted, thick] (A1) to[bend right=35] (A3);
    \draw[dotted, thick] (A2) to[bend right=35] (A4);
  \end{tikzpicture}
  \hspace{1cm}
  %
  \begin{tikzpicture}[scale=1.562,baseline=(anchor),
      every node/.style={font=\small},
      nd/.style ={circle,   draw, thick, minimum size=5mm, inner sep=0pt},
      rt/.style ={rectangle,draw, thick, minimum size=5mm, inner sep=0pt}]
    \node[rt] (B0) at (0,0) {0};
    \node[nd] (B1) at ( 90:1.2) {1};
    \node[nd] (B2) at ( 18:1.2) {2};
    \node[nd] (B3) at (306:1.2) {3};
    \node[nd] (B4) at (234:1.2) {4};
    \node[nd] (B5) at (162:1.2) {5};
    \draw (B0) -- (B1);
    \draw (B0) -- (B2);
    \draw (B0) -- (B3);
    \draw (B0) -- (B4);
    \draw (B0) -- (B5);
    \draw[very thick] (B1) -- (B2);
    \draw[very thick] (B2) -- (B3);
    \draw[very thick] (B3) -- (B4);
    \draw[very thick] (B4) -- (B5);
    \draw[very thick] (B5) -- (B1);
    \draw[dotted, thick] (B1) -- (B3);
    \draw[dotted, thick] (B1) -- (B4);
    \draw[dotted, thick] (B2) -- (B4);
    \draw[dotted, thick] (B2) -- (B5);
    \draw[dotted, thick] (B3) -- (B5);
    \coordinate (anchor) at (0,0.1);
  \end{tikzpicture}
  \hspace{1cm}
  %
  \begin{tikzpicture}[scale=1.43,baseline=(C0.center),
      every node/.style={font=\small},
      nd/.style ={circle,   draw, thick, minimum size=5mm, inner sep=0pt},
      rt/.style ={rectangle,draw, thick, minimum size=5mm, inner sep=0pt}]
    \node[rt] (C0) at (0,0) {0};
    \node[nd] (C1) at ( 90:1.2) {1};
    \node[nd] (C2) at ( 30:1.2) {2};
    \node[nd] (C3) at (330:1.2) {3};
    \node[nd] (C4) at (270:1.2) {4};
    \node[nd] (C5) at (210:1.2) {5};
    \node[nd] (C6) at (150:1.2) {6};
    \draw (C0) -- (C1);
    \draw (C0) -- (C2);
    \draw (C0) -- (C3);
    \draw (C0) -- (C4);
    \draw (C0) -- (C5);
    \draw (C0) -- (C6);
    \draw[very thick] (C1) -- (C2);
    \draw[very thick] (C2) -- (C3);
    \draw[very thick] (C3) -- (C4);
    \draw[very thick] (C4) -- (C5);
    \draw[very thick] (C5) -- (C6);
    \draw[very thick] (C6) -- (C1);
    \draw[dotted, thick] (C1) -- (C3);
    \draw[dotted, thick] (C3) -- (C5);
    \draw[dotted, thick] (C5) -- (C1);
    \draw[dotted, thick] (C2) -- (C4);
    \draw[dotted, thick] (C4) -- (C6);
    \draw[dotted, thick] (C6) -- (C2);
    \draw[dotted, thick] (C1) to[bend right=25] (C4);
    \draw[dotted, thick] (C2) to[bend right=25] (C5);
    \draw[dotted, thick] (C3) to[bend right=25] (C6);
  \end{tikzpicture}
  \caption{Wheel-graph instances for $n=4,5,6$. Depot edges (cost $c_e=1$) are thin,
           boundary edges (cost $c_e=1$) are bold,
           and diagonal edges (cost $c_e=2$) are dotted}
  \label{fig:wheel_graphs}}
\end{figure}

We consider the OR policy, under which we define the \textit{wheel-graph instances} by taking demands $\xi_i \sim \text{Bern}(\mu)$, $\mu \in (0,1)$, for each $i \in N$, a single vehicle, i.e., $M=\{1\}$, with capacity $Q=n{-}1$, and penalty parameters $b^F=b^P=0$. The recourse costs are then:
\[
  c_i^F = 2 \;\;\forall\, i\in N, \qquad
  c_e^P = c_{0,i}+c_{0,j}-c_{i,j} = 2-c_{i,j} =
  \begin{cases} 1, & \text{if } e=\{i,j\}\in E_1,\\
                 0, & \text{if } e=\{i,j\}\in E_2.
  \end{cases}
\]

Every feasible solution consists of a single route $(0,p,0)$, for $p=(i_1,\dots,i_n)$ some permutation of~$N$. Call such a route \emph{boundary} if its $n-1$ customer edges $\{i_j, i_{j+1}\}$ all lie in~$E_1$, and \emph{nonboundary} otherwise. There are $n$ boundary routes; for each $k \in \{1,\dots,n\}$, we denote by $(0,p^{(k)},0)$ the one omitting edge~$e_k$, and by~$x^{(k)}$ the corresponding solution.

\begin{lemma}\label{lem:wheel}
On wheel-graph instances, (i) every boundary route has expected recourse cost $q^*:=\mu^{n-1}\min\{2\mu, 1\}$; (ii) every nonboundary route has expected recourse cost~$0$; and (iii) the optimal solutions are exactly the $n$~boundary routes, with optimal value $Z^* = n+1+q^* \in (n+1,n+2)$.
\end{lemma}

\begin{proof}
(i) Let $(0,p,0)=(0,i_1,\dots,i_n,0)$ be a boundary route. The recourse cost-to-go after serving customer~$i_{n-1}$ is $F_{i_{n-1}}^{p}(q)=0$ if $q \geq 1$ and $F_{i_{n-1}}^{p}(0) = \min\{c^F_{i_n}\mu,\, c^P_{i_{n-1},i_n}\} = \min\{2\mu, 1\}$ otherwise. Since $Q=n-1$ and each demand lies in $\{0,1\}$, no failure can occur before the last customer, and the residual capacity after serving customer~$i_j$, $j \leq n-2$, is at least $n-1-j \geq 1$. The cost-to-go after serving $i_j$ is thus at most $\mu^{n-1-j}\cdot F_{i_{n-1}}^{p}(0) \leq \mu < 1 = c^P_{i_j,i_{j+1}}$, so a preventive return is never optimal and thus $\mathcal{Q}^{\text{OR}}_{p} = \mu^{n-1} \cdot F_{i_{n-1}}^{p}(0) = q^*$.

(ii) A nonboundary route traverses at least one diagonal edge $e \in E_2$, with $c^P_e = 0$. A costless preventive return at this edge restores capacity~$Q = n-1$, which suffices for all remaining customers. Hence $\mathcal{Q}^{\text{OR}}_p = 0$.

(iii) A boundary route uses $n-1$ boundary edges of cost~$1$ and two depot edges of cost~$1$, for a first-stage cost of~$n+1$. A nonboundary route uses at least one diagonal edge of cost~$2$ in place of a boundary edge of cost~$1$, for a first-stage cost of at least~$n+2$. Since $q^* \in (0,1)$, the total cost of every boundary route is $Z^* = n+1+q^* \in (n+1,n+2)$, strictly less than the total cost $\geq n+2$ of any nonboundary route.
\end{proof}

\paragraph{Convergence results.} Let RMP denote the relaxed master problem obtained from~MP by relaxing all P-cuts~\eqref{DL:p-cut} and S-cuts~\eqref{DL:s-cut}. Without additional cutting planes, RMP has optimal value~$n+1 < n+1+q^* = Z^*$, since any boundary route with $\theta = 0$ is feasible. We now show that a single E-cut suffices to close this gap (Theorem~\ref{thm:one-ecut}), whereas at least~$n$ cuts from the families of valid inequalities proposed by~\cite{parada2024disaggregated} and~\cite{hoogendoorn2023improved} are needed (Theorems~\ref{thm:n-DL} and~\ref{thm:n-HS}).

\begin{theorem}\label{thm:one-ecut}
On wheel-graph instances, the optimal value of~RMP augmented with a single suitably chosen E-cut~\eqref{DL:e-cut} is~$Z^*$.
\end{theorem}

\begin{proof}
Take $S=N$ and $E_S=E_1$. Every feasible route using only boundary edges is a boundary route, so $\bar{\mathcal{L}}(E_1,1)=q^*$. The corresponding E-cut reads:
\begin{equation}\label{eq:wheel-ecut}
  \sum_{i\in N}\theta_i
  \geq
  q^*\Bigl(\sum_{e\in E_1} x_e - n + 2\Bigr).
\end{equation}
Let $x$ be a feasible solution and $(0,p,0)$ the route it forms. The right-hand side of~\eqref{eq:wheel-ecut} evaluates to~$q^*$ if $(0,p,0)$ is a boundary route, and is nonpositive otherwise. Therefore, by Lemma~\ref{lem:wheel}, $\min\{\sum_{i\in N}\theta_i : \eqref{eq:wheel-ecut},\; \theta \geq 0\} = \mathcal{Q}^{\text{OR}}(x)$.
\end{proof}

\begin{theorem}\label{thm:n-DL}
On wheel-graph instances, the optimal value of~RMP augmented with any collection of fewer than~$n$ P-cuts~\eqref{DL:p-cut} and S-cuts~\eqref{DL:s-cut} is strictly less than~$Z^*$.
\end{theorem}

\begin{proof}
Any strict subset $S \subset N$ satisfies $\sum_{i \in S}\xi_i \leq |S| \leq n-1 = Q$ with probability one, so $\mathcal{L}(S,1) = 0$. Moreover, any nonboundary route $(0,p,0)$ covers $N$ and has $\mathcal{Q}^{\text{OR}}_p = 0$, so $\mathcal{L}(N,1) = 0$. Every S-cut is thus trivial. The P-cut of path~$p$ has coefficient $\mathcal{Q}^{\text{OR}}_p = q^*$ if $(0,p,0)$ is a boundary route, and $\mathcal{Q}^{\text{OR}}_p = 0$ otherwise. The P-cut \eqref{DL:p-cut} for~$p=p^{(k)}$ evaluates to~$q^*$ at~$x^{(k)}$, since $p^{(k)}(x^{(k)})=n-1=|N(p^{(k)})|-1$, but to~$0$ at any~$x^{(j)}$, $j \neq k$, since $p^{(k)}(x^{(j)})=n-2$. Each nontrivial P-cut thus enforces $\sum_{i \in N}\theta_i \geq q^*$ at a single optimal solution. If fewer than~$n$ are added, some optimal solution~$x^{(k)}$ remains uncovered, and $(x^{(k)}, \theta = 0)$ is feasible with objective value $n+1 < Z^*$.
\end{proof}

We now turn to the valid inequalities from \cite{hoogendoorn2023improved}. First, we note that, for single-route instances where there exists a feasible solution with zero recourse cost (in particular, for wheel-graph instances), all the families of valid inequalities therein are subsumed by the partial route-split inequalities \eqref{eq:prs}.

\begin{remark}\label{rem:subsumption}
For $M=\{1\}$ and $\mathcal{L}(N,1) = 0$, the partial route-split inequalities subsume all the other valid inequalities from~\cite{hoogendoorn2023improved}:
\begin{enumerate}\renewcommand{\labelenumi}{(\roman{enumi})}
  \item A \emph{route-split inequality} \cite[equation~(15)]{hoogendoorn2023improved} is a special case of partial route-split inequality where $h$ is fully structured; see \cite[Section~5.1]{hoogendoorn2023improved}.
  \item Since $M=\{1\}$, the \emph{partial route inequality} \cite[equation~(12)]{hoogendoorn2023improved} associated with a partial route $h$ is inactive at any feasible solution if $N(h)\neq N$. Otherwise, since $\mathcal{L}(N,1) = 0$, it reduces to $\sum_{i \in N}\theta_i \geq P_h\, W_h(x)$, which is implied by~\eqref{eq:prs} together with $\sum_{i \in N}\theta_i \geq \theta_{v(h)}$; see \cite[Section~5.2]{hoogendoorn2023improved}.
  \item Since $M=\{1\}$, a \emph{multi-route-split inequality} \cite[equation~(18)]{hoogendoorn2023improved} associated with a set $S \neq N$ is inactive at any feasible solution. For $S = N$ and a single route, it reduces to a partial route inequality; see \cite[Section~5.3]{hoogendoorn2023improved}.
\end{enumerate}
\end{remark}

For each $k \in \{1,\dots,n\}$, let $h^{(k)} = (\{0\}, \{p^{(k)}_1\}, \ldots, \{p^{(k)}_n\}, \{0\})$ denote the fully structured partial route corresponding to the boundary route~$(0,p^{(k)},0)$. Lemma~\ref{lem:prs-characterization} shows that these are the only partial routes that produce nontrivial partial route-split inequalities on wheel-graph instances.

\begin{lemma}\label{lem:prs-characterization}
On wheel-graph instances, the only nontrivial partial route-split inequalities~\eqref{eq:prs} are those for $h = h^{(k)}$, $k \in \{1,\dots,n\}$.
\end{lemma}

\begin{proof}
Fix a partial route $h = (U_0, U_1, \ldots, U_b)$. If $N(h) \neq N$, no feasible route adheres to $h$ since $M=\{1\}$, so $W_h(x) \leq 0$ at every feasible solution~$x$, and the inequality is never active. Now suppose $N(h)=N$. If $h$ is fully structured and corresponds to a nonboundary route, its unique adhering route has recourse~$0$ by Lemma~\ref{lem:wheel}, so $P_h=0$. If $h$ has an unstructured component~$U_a$, $a \in \{1,\dots,b-1\}$, we show that $P_h=0$ by exhibiting a nonboundary route that adheres to $h$. We consider three cases. (i)~If $|U_a|\ge 3$, then $U_a$ contains two customers that are not adjacent on $E_1$; they can be visited consecutively within $U_a$, resulting in a nonboundary route that adheres to $h$. (ii)~If $|U_a|=2$ and $U_a \in E_2$, every adhering route traverses the diagonal edge~$U_a$. (iii)~If $|U_a|=2$ and $U_a = \{i_1, i_2\} \in E_1$, take any customer $j \in (U_{a-1} \cup U_{a+1}) \setminus \{0\}$. Because $n \geq 4$ and $\{i_1, i_2\} \in E_1$, $i_1$ and $i_2$ cannot both be adjacent to $j$ on the cycle $E_1$, so at least one of $\{j, i_1\}, \{j, i_2\}$ lies in~$E_2$. Ordering $U_a$ so that this diagonal edge is visited yields a nonboundary adhering route.

The remaining case is $h = h^{(k)}$, $k \in \{1,\dots,n\}$, with $R_{h^{(k)}} = \{(0,p^{(k)},0)\}$ and $P_{h^{(k)}} = \mathcal{Q}^{\text{OR}}_{p^{(k)}} = q^*$. The resulting partial route-split inequality~\eqref{eq:prs} enforces $\theta_{v(h^{(k)})} \geq q^*$ at~$x^{(k)}$ and is inactive at every other feasible solution.
\end{proof}

\begin{theorem}\label{thm:n-HS}
On wheel-graph instances, the optimal value of~RMP augmented with any collection of fewer than~$n$ valid inequalities from~\cite{hoogendoorn2023improved} is strictly less than~$Z^*$.
\end{theorem}

\begin{proof}
Since $M=\{1\}$ and $\mathcal{L}(N,1) = 0$, Remark~\ref{rem:subsumption} implies that it suffices to consider partial route-split inequalities~\eqref{eq:prs}. By Lemma~\ref{lem:prs-characterization}, each nontrivial partial route-split inequality enforces $\theta_{v(h^{(k)})} \geq q^*$ at a unique optimal solution~$x^{(k)}$, $k \in \{1,\dots,n\}$. If fewer than~$n$ are added, some optimal solution~$x^{(k)}$ remains uncovered, and $(x^{(k)}, \theta = 0)$ is feasible with objective value $n+1 < Z^*$.
\end{proof}
}


\subsection{\rev{Edge-set selection strategy}} \label{sec:E-cuts:strategy}

Building E-cuts \eqref{DL:e-cut} involves a trade-off, as restricting the edge set $E_S$ increases the bound $\bar{\mathcal{L}}(E_S, m)$, but reduces the set of solutions for which the E-cut is active. We propose a simple edge-set selection strategy \rev{that proceeds by removing edges from a complete subgraph $E(S)$.} 

Let $x^{\nu}$ be a first-stage solution, either integer or fractional, and $S \subseteq N$ an arbitrary set of customers. \rev{Let $\bar{c}^P(x^{\nu}, S) := \min\{c^P_e : e \in E(S), \  x^{\nu}_e > 0\}$ denote the minimal preventive return cost among the edges of $E(S)$ active at $x^{\nu}$. We obtain the E-cut associated with the pair $(x^{\nu}, S)$ by excluding from $E(S)$ the edges whose preventive return cost is strictly less than $\bar{c}^P(x^{\nu}, S)$, which yields the edge-set $E_S = E_S(x^{\nu}) := \{e \in E(S) \mid c^P_e \geq \bar{c}^P(x^{\nu}, S)\}$.} This selection ensures that the resulting E-cut \eqref{DL:e-cut} is active for the current solution. 

\rev{We illustrate this strategy on the wheel-graph instances of Section~\ref{sec:E-cuts:wheel}. Suppose RMP is solved to optimality. Without valid inequalities on the recourse, the first-stage cost alone determines the solution, and the master problem returns a boundary route $(0,p^{(k)},0)$, $k \in \{1,\dots,n\}$, with objective value $n+1$ (Lemma~\ref{lem:wheel}). Since all the active customer edges are boundary edges $e \in E_1$ with $c^P_e = 1$, the edge-set selection strategy for $x^{\nu} = x^{(k)}$ and $S=N$ yields $c^P_{e^{\nu}} = 1$ and $E_S(x^{\nu}) = \{e \in E(N) \mid c^P_e \geq 1\} = E_1$. This is precisely the edge set used in Theorem~\ref{thm:one-ecut}. The resulting E-cut, with $\mathcal{L}_{E_S} = \bar{\mathcal{L}}(E_1,1) = q^*$, thus provides an exact representation of the OR recourse function.}

\subsection{Lower bounds on the OR recourse function}\label{sec:bounds}

The E-cut~\eqref{DL:e-cut} requires a lower bound $\mathcal{L}_{E_S}$ on $\bar{\mathcal{L}}(E_S, m_{E_S})$, the smallest recourse cost achievable by an admissible partition of a set of customers $S$ into paths formed of edges in $E_S$. In this section, we develop such lower bounds under the OR policy. \rev{By Proposition~\ref{prop:E-cuts-generalize}, the S-cut is the special case $E_S = E(S)$ of the E-cut, with $\mathcal{L}(S, m_S) = \bar{\mathcal{L}}(E_S, m_{E_S})$, so the same constructions yield lower bounds $\mathcal{L}_S$ for the S-cut~\eqref{DL:s-cut}.} In Section~\ref{subsubsec:RecLBMV_all}, we present a general method that is applicable to any demand distribution. In Section~\ref{subsubsec:RecLBMV_poisson}, this method is specialized to the case of Poisson-distributed demands.
 
\subsubsection{General lower bound} \label{subsubsec:RecLBMV_all}
Consider a set of customers $S \subseteq N$, \rev{an edge set $E_S \subseteq E(S)$}, and a number of routes $m_{E_S} \in \{1,\dots,\bar{m}(E_S)\}$. The lower bound we derive on $\bar{\mathcal{L}}(E_S, m_{E_S})$ depends on an indexing $i_1, \dots, i_{|S|}$ of the customers of $S$; this indexing is a free parameter, and our specific choice is described at the end of this subsection.

Fix a customer indexing, and consider any partition $\pi \in \bar{\Pi}(E_S, m_{E_S})$. Since the fleet is homogeneous, we may arbitrarily index the paths of $\pi$ and their vehicles as $k \in \{1,\dots,m_{E_S}\}$. We do so such that customer $i_j$ is visited by some vehicle $k \leq j$; equivalently, every customer visited by vehicle $k$ belongs to $\{i_k,\dots,i_{|S|}\}$. 
 
From there, $c^F_{(k)} := \min_{j \geq k} c^F_{i_j}$ is the smallest failure cost accessible to vehicle $k$ and \rev{$c^P_{(k)} := \min_{\{i_j,i_\ell\} \in E_S} \{ c^P_{i_j,i_\ell} : j,\ell \geq k \}$ is the smallest preventive return cost accessible to vehicle $k$, with $c^P_{(k)} := +\infty$ if no such edge exists}. Finally, we denote by $c^R_{(k)} := \min\{c^F_{(k)},c^P_{(k)}\}$ the cost of the cheapest recourse action vehicle $k$ can perform.

Let $\bar\mu \in \mathbb{R}_{>0}$ be such that each $\mu_i$, $i \in S$, is an integer multiple of $\bar\mu$; if $\{\mu_i\}_{i \in S} \subseteq \mathbb{N}$, $\bar\mu$ may be taken as $\mathrm{GCD}(\{\mu_i\}_{i \in S})$. \rev{Define the total expected demand of the customers of $S$, the number of groups of $\bar{\mu}$ units this yields, and the maximum number of groups of $\bar{\mu}$ units that can be assigned to a route while respecting the load factor requirement as:
\[
D := \sum_{i \in S}\mu_i, \qquad \bar D := D/\bar\mu, \qquad \bar n := \lfloor fQ/\bar\mu \rfloor.
\]}
\rev{For each $d \in \{0,\bar\mu,\dots,\bar n \bar\mu\}$}, let $\rho(d)$ satisfy
\begin{equation}\label{eq:rho_lower_envelope}
\rho(d)
\leq
\mathds{P}\!\left(\sum_{i \in S'} \xi_i > Q\right)
\quad
\rev{\text{for every } S' \subseteq S \text{ such that } \sum_{i \in S'}\mu_i = d,}
\end{equation}
\rev{with $\rho(d) := 0$ when no such $S'$ exists. The quantity $\rho(d)$ is a lower bound on the probability that a vehicle assigned a set of customers with total expected demand $d$} must perform at least one recourse action (either a failure or a preventive restocking trip) to complete its route. Since every recourse action of vehicle $k$ costs at least $c^R_{(k)}$, the quantity $g_k(d) := \rho(d) \, c^R_{(k)}$
is a lower bound on the expected recourse cost of vehicle $k$ under the OR policy whenever the total expected demand of the customers it visits is $d$.

A lower bound on $\bar{\mathcal{L}}(E_S, m_{E_S})$ can therefore be obtained by partitioning the expected demands $\{\mu_i\}_{i \in S}$ over the $m_{E_S}$ vehicles to minimize $\sum_{k=1}^{m_{E_S}} g_k(d_k)$. \rev{We relax this set-partitioning problem to the assignment problem \eqref{genLB:obj}--\eqref{genLB:demand}, in which a multiple of $\bar{\mu}$ units of expected demand must be assigned to each vehicle rather than a specific subset of $\{\mu_i\}_{i \in S}$.}
\begingroup
\allowdisplaybreaks
\begin{align}
    \mathcal{L}^1_{E_S} := \min & \sum_{k=1}^{m_{E_S}} \sum_{\ell=0}^{\bar n} g_k(\ell \bar\mu) \, y_{k\ell} \label{genLB:obj} \\
    \text{s.t.} & \sum_{\ell=0}^{\bar n} y_{k\ell} = 1, & \rev{\forall} k \in \{1,\dots,m_{E_S}\}, \label{genLB:assignment} \\
    & \sum_{k=1}^{m_{E_S}} \sum_{\ell=0}^{\bar n} \ell \, y_{k\ell} = \bar D, \label{genLB:demand} \\
    & y_{k\ell} \in \{0,1\}, & \rev{\forall} k \in \{1,\dots,m_{E_S}\},\, \rev{\forall} \ell \in \{0,\dots,\bar n\}. \label{genLB:vars_y}
\end{align}
\endgroup
\rev{Constraint~\eqref{genLB:assignment} assigns one demand level to each vehicle, and constraint~\eqref{genLB:demand} requires $\bar D$ groups of $\bar\mu$ units to be assigned in total.} The separability of the objective over vehicles enables a dynamic-programming solution. For each $k \in \{1,\dots,m_{E_S}\}$, the minimum cost of distributing $\bar d \in \{0,\dots,\bar D\}$ groups of $\bar\mu$ units of expected demand into vehicles $1$ to $k$ is given by:
\begin{equation}\label{LB_kd_DP}
G_{k}(\bar{d}) := \begin{cases}
        \min\limits_{y \in \{0,1,\dots,\min\{\bar{d}, \bar{n}\}\}}
        \bigl(g_{k}(y\bar{\mu}) + G_{k-1}(\bar{d}-y)\bigr),
        & \text{if } k \in \{2,\dots,m_{E_S}\},\\[2pt]
        g_{1}(\bar{d}\bar{\mu}),
        & \text{if } k = 1 \text{ and } \bar{d} \leq \bar{n},\\
        +\infty,
        & \text{if } k = 1 \text{ and } \bar{d} > \bar{n}.
    \end{cases}
\end{equation}
In particular, \rev{we show in Proposition \ref{prop:LB1}} that the cost $\mathcal{L}^1_{E_S} = G_{m_{E_S}}(\bar D)$ \rev{of distributing all $\bar D$ groups across the $m_{E_S}$ vehicles yields a valid cut coefficient.}
 
\begin{proposition}\label{prop:LB1}
For $\mathcal{Q}=\mathcal{Q}^{\mathrm{OR}}$, $\mathcal{L}^1_{E_S}$ is a valid coefficient for the E-cut~\eqref{DL:e-cut}, i.e., $\mathcal{L}^1_{E_S} \leq \bar{\mathcal{L}}(E_S,m_{E_S})$. \rev{In particular, $\mathcal{L}^1_S := \mathcal{L}^1_{E(S)} \leq \mathcal{L}(S,m_S)$ is a valid coefficient for the S-cut~\eqref{DL:s-cut}.}
\end{proposition}

\begin{proof}\rev{
    Fix a partition $\pi \in \bar\Pi(E_S, m_{E_S})$. Index the vehicles serving the paths of $\pi$ as $\{p_k\}_{k=1}^{m_{E_S}}$ so that customer $i_j$ is visited by some vehicle $k \leq j$. Write $d_k = \sum_{i \in N(p_k)} \mu_i$. By construction, $\mathcal{Q}^{\mathrm{OR}}_{p_k} \geq g_k(d_k)$. Since each $\mu_i$ is an integer multiple of $\bar\mu$ and $d_k \leq fQ$, the assignment $y_{k\ell} = 1$ if $\ell\bar\mu = d_k$ and $y_{k\ell} = 0$ otherwise is feasible for~\eqref{genLB:assignment}--\eqref{genLB:vars_y}, so $\sum_k g_k(d_k) \geq \mathcal{L}^1_{E_S}$. Combining with the per-vehicle bound yields $\sum_k \mathcal{Q}^{\mathrm{OR}}_{p_k} \geq \mathcal{L}^1_{E_S}$, and since $\pi$ was arbitrary, $\bar{\mathcal{L}}(E_S, m_{E_S}) \geq \mathcal{L}^1_{E_S}$. }
\end{proof}

\paragraph{Implementation.}
We select the customer indexing greedily. At each iteration $k = 1,\dots,m_{E_S}-1$, we assign to index $i_k$ the unassigned customer of $S$ that maximizes the minimum recourse cost $c^R_{(k+1)}$ for vehicle $k+1$, breaking ties by smallest failure cost. The remaining customers are indexed arbitrarily. 

We apply $\mathcal{L}^1_{E_S}$ only to instances in which the demands are identically distributed with expectation $\mu$. For this case, we can select $\bar\mu = \mu$ and precompute the probability of mandatory recourse action as $\rho(t\mu) = \mathds{P}\!\left(\sum_{i=1}^t \xi_i > Q\right)$ for each number of customers $t \in \{0, \dots, \lfloor fQ/\mu \rfloor\}$ that can appear on a route.

\subsubsection{Distribution-specific lower bound} \label{subsubsec:RecLBMV_poisson}
\rev{We refine the bound of Section~\ref{subsubsec:RecLBMV_all} for Poisson-distributed demands by replacing the per-vehicle lower bounding function $g_k$ by an alternative bound $\widetilde{g}_k$. This bound exploits the additivity property of the Poisson distribution and is based on the expected recourse cost-to-go under the OR policy rather than on the probability of a mandatory recourse action.}
 
Consider a path $p_k$ visited by vehicle $k$ with $N(p_k) \subseteq \{i_k, \dots, i_{|S|}\}$. Decompose each customer $i \in N(p_k)$ into $\mu_i / \bar\mu$ \emph{subcustomers} with i.i.d.\ Poisson$(\bar\mu)$ demands; by additivity of the Poisson distribution, the total demand of the subcustomers of $i$ has the same distribution as $\xi_i$. Let $p'_k$ be the path obtained from $p_k$ by visiting all subcustomers of the same group consecutively. The OR policy on $p_k$ is a restricted version of the OR policy on $p'_k$ that prohibits preventive returns between subcustomers of the same group, so $\mathcal{Q}^{\mathrm{OR}}_{p'_k} \leq \mathcal{Q}^{\mathrm{OR}}_{p_k}$. Replacing each failure cost on $p'_k$ by $c^F_{(k)}$ and each preventive return cost by $c^P_{(k)}$ in the cost-to-go recursion~\eqref{def:cost-to-go} yields a lower bound on $\mathcal{Q}^{\mathrm{OR}}_{p'_k}$ that depends on path $p'_k$ only through the number of subcustomers it visits. Concretely, for a residual capacity $q \in \{0, \dots, Q\}$ and $\bar d \in \{1, \dots, \bar n\}$ subcustomers remaining, the resulting cost-to-go is:
\begin{equation}
\label{F_tilde_DP}
\hspace{-0.2cm}
\widetilde{F}^k_{\bar{d}}(q) := \min
\begin{cases}
\widetilde{H}^k_{\bar{d}}(q) := & \sum\limits_{s=0}^{+\infty} \left[ c^F_{(k)}\Psi(s,q) + \widetilde{F}^k_{\bar{d}-1}\!\left(\Psi(s,q)Q + q - s\right) \right] \rho^s, \\
\widetilde{H}^{*k}_{\bar{d}} := & c^P_{(k)} + \widetilde{H}^k_{\bar{d}}(Q),
\end{cases}\hspace{-0.19cm}
\end{equation}
with boundary condition $\widetilde{F}^k_0(q) := 0$ for all $q \in \{0, \dots, Q\}$, and $\rho^s$ the Poisson$(\bar\mu)$ probability mass at $s$. Setting $\widetilde{g}_k(\bar d \bar\mu) := \widetilde{F}^k_{\bar d}(Q)$ yields a per-vehicle lower bound that satisfies $\widetilde{g}_k(d_k) \leq \mathcal{Q}^{\mathrm{OR}}_{p_k}$ for every feasible path $p_k$ with total expected demand $\sum_{i \in N(p_k)} \mu_i = d_k$. Let $\mathcal{L}^2_{E_S}$ denote the optimal value of the assignment problem~\eqref{genLB:obj}--\eqref{genLB:vars_y} with $g_k$ replaced by $\widetilde{g}_k$. Replacing $g_k$ by $\widetilde{g}_k$ in the proof of Proposition~\ref{prop:LB1} yields Proposition~\ref{prop:LB2}.
 
\begin{proposition}\label{prop:LB2}
For $\mathcal{Q}=\mathcal{Q}^{\mathrm{OR}}$, if $\xi_i \sim \mathrm{Poisson}(\mu_i)$ for all $i \in S$, then $\mathcal{L}^2_{E_S}$ is a valid coefficient for the E-cut~\eqref{DL:e-cut}, i.e., $\mathcal{L}^2_{E_S} \leq \bar{\mathcal{L}}(E_S, m_{E_S})$. \rev{In particular, $\mathcal{L}^2_S := \mathcal{L}^2_{E(S)} \leq \mathcal{L}(S, m_S)$ is a valid coefficient for the S-cut~\eqref{DL:s-cut}.}
\end{proposition}

\paragraph{Implementation.}
We compute $\mathcal{L}^2_{E_S}$ using the same greedy indexing as for $\mathcal{L}^1_{E_S}$ and the dynamic-programming recursion~\eqref{LB_kd_DP} with $g_k$ replaced by $\widetilde{g}_k$. The values $\widetilde{g}_k(\bar d \bar\mu)$, $\bar d \in \{0, \dots, \bar n\}$, are precomputed for each vehicle $k$ via the cost-to-go recursion \eqref{F_tilde_DP}.

\section{Implementation of the DL-shaped algorithm} \label{sec:implementation}

We now describe our B\&C implementation of the DL-shaped method. At each node of the B\&B tree, we solve the LP relaxation of the current master problem and separate rounded capacity inequalities \eqref{cap_cut}, P-cuts \eqref{DL:p-cut}, S-cuts \eqref{DL:s-cut}, and E-cuts \eqref{DL:e-cut} from the resulting solution $(x^\nu, \{\theta_i^\nu\}_{i \in N})$. The rounded capacity inequalities are separated heuristically using the CVRPSEP package \cite{Lysgaard2003}. The candidate DL-shaped cuts we separate then depend on the structure of the solution $x^\nu$. We evaluate (i) the S-cuts and E-cuts of the customer sets returned by CVRPSEP, (ii) the S-cuts and E-cuts of the connected components of $x^\nu$ when it is fractional, together with the P-cut of any such component that forms a path, and (iii) all three families of cuts for the subpaths of $\mathcal{P}(x^\nu)$ when $x^\nu$ is integer and no rounded capacity inequality was added in this iteration. In case~(iii), we retain only the six most violated cuts of each type. Algorithm~\ref{alg:dl-shaped} summarizes the separation procedure.

\begin{algorithm}[H]
\caption{Separation procedure at a B\&B node}\label{alg:dl-shaped}
\begin{algorithmic}[1]
\State Solve the LP relaxation of the master problem; obtain $(x^\nu, \{\theta_i^\nu\}_{i \in N})$
\State Separate violated rounded capacity inequalities \eqref{cap_cut} via CVRPSEP
\For{each set $S$ associated with a violated rounded capacity inequality}
    \State Add the S-cut and E-cut of $S$ if violated
\EndFor
\If{$x^\nu$ is fractional}
    \For{each connected component $S$ of $x^\nu$}
        \State Add the S-cut and E-cut of $S$ if violated
        \If{$|\{e \in E(S) : x^{\nu}_e > 0\}| = |S|-1$}
            \State Add the P-cut of the path covering $S$ if violated
        \EndIf
    \EndFor
\ElsIf{$x^\nu$ is integer and no rounded capacity inequality was added}
    \For{each subpath $p$ of $\mathcal{P}(x^\nu)$ with $S = N(p)$}
        \State Evaluate the P-cut of $p$, the S-cut of $S$, and the E-cut of $E_S(x^\nu)$
    \EndFor
    \State Add the six most violated P-cuts, S-cuts, and E-cuts
\EndIf
\end{algorithmic}
\end{algorithm}

The procedure ends when all the nodes of the B\&B tree have been explored; the incumbent solution is then optimal. \rev{The initial relaxed master problem is the DL-shaped reformulation MP without P-cuts, and equipped with an initial pool of S-cuts.} This pool is composed of all the nontrivial S-cuts for sets of customers $S\subseteq N$ of cardinality $|S| \in \{2,3,4\}$ for instances with $n \leq 32$ customers, and of cardinality $|S| \in \{2,3\}$ for larger instances. For these S-cuts, we compute $\mathcal{L}(S,1)$ by enumeration and take $\mathcal{L}_S = \mathcal{L}(S,1)$. For the S-cuts and E-cuts generated in callbacks for a set $S \subseteq N$, we take $m_S = m_{E_S} = \Big\lceil \frac{1}{\lfloor fQ/\bar{\mu} \rfloor \bar{\mu}}\sum_{i \in S}\mu_i \Big\rceil$, where $\bar{\mu}=\text{GCD}(\{\mu_i\}_{i \in S})$. \rev{For the E-cuts, the edge-set $E_S$ is selected according to the procedure described in Section \ref{sec:E-cuts:strategy}. We use the Poisson-specific lower bound $\mathcal{L}^2_{E_S}$ of Proposition~\ref{prop:LB2} for instances with Poisson-distributed demands, and the general lower bound $\mathcal{L}^1_{E_S}$ of Proposition~\ref{prop:LB1} otherwise. The corresponding S-cut coefficients $\mathcal{L}^2_{S}$ and $\mathcal{L}^1_{S}$ are obtained as the special case $E_S = E(S)$.
}

\section{Computational results} \label{sec:experiments}
\rev{This section presents our computational study. The instances and computational setup are introduced in Section \ref{sec:experiments:instances}.} In Section \ref{sec:experiments:e-cuts}, we evaluate the impact of our new valid inequalities on the overall performance of the DL-shaped algorithm. In Section \ref{sec:experiments:existing}, we compare our algorithm to state-of-the-art methods from the literature on standard benchmark instances of the VRPSD under the OR policy. In Section \ref{sec:experiments:characteristics}, we report results on a new set of instances and discuss instance characteristics that affect the performance of our algorithm. 

\subsection{Instances and experimental setup} \label{sec:experiments:instances}
To compare our algorithm with methods from the literature, we use two existing sets of benchmark instances. The first one, proposed by \cite{louveaux2018exact}, is composed of 32 instances in which the demands are identically distributed according to a discrete symmetric triangular distribution. The second set was generated based on 90 instances (sets A, B, E, F, and P) of the CVRPLIB repository \cite{uchoa2017new}. The customer demands follow the Poisson distribution with expected values matching the deterministic demands from the original CVRPLIB instances.

In addition, we adapt two benchmark sets from \cite{jabali2014partial} and \cite{parada2024disaggregated}. These instances were originally designed for the VRPSD under the DTD policy, and use normally distributed demands. The set of \cite{jabali2014partial} contains 270 instances, where the number of customers ranges from 40 to 80, and the number of vehicles ranges from two to four. That of \cite{parada2024disaggregated} comprises 1,980 instances with the number of customers and the number of vehicles ranging from 20 to 120 and two to seven. We modify these instances by assuming that the customer demands are Poisson-distributed, with the same expectations as in the original instances. As the distance matrix provided in existing instances may violate the triangle inequality, we use the Floyd-Warshall algorithm to enforce this assumption in our experiments. To facilitate the comparison with previous works, the penalty parameters are set to $b^F = b^P = 0$ for all the instances, and the load factor parameter is set to $f=1$, unless otherwise specified.

We evaluate all four variants of the VRPSD defined in Section~\ref{subsec:VRPSDform}. Prior works have predominantly benchmarked under the ECCs, as relaxing them yields substantially harder instances. In particular, by Theorem~2 of \cite{yang2000stochastic}, the basic variant under the OR policy always admits an optimal solution comprising a single route. This variant is notoriously challenging for exact methods~\cite{florio2020optimal}, as the decomposition benefit of column generation is lost in the case of single-route solutions, and a large number of feasible routes makes the cutting-plane approximation of the recourse function less efficient in B\&C methods. Only \cite{hoogendoorn2025evaluation} have previously solved the VRPSD in the FRC and basic variants. 

The DL-shaped method was implemented in C{\footnotesize ++} with CPLEX version 22.1 using callbacks, and compiled using g{\tiny ++}. It was run single-threaded on a computing cluster node with an AMD EPYC™ Rome 7532 2.4 GHz processor and 48 GB of RAM. We used an adaptive large neighborhood search algorithm \cite{pisinger2007general} to warm-start the model. A time limit of 1 hour per instance was applied for all the experiments.

\subsection{Impact of the new valid inequalities} \label{sec:experiments:e-cuts}
In this section, we evaluate the impact of our new valid inequalities on the overall performance of the DL-shaped algorithm. We solve the 270 instances adapted from \cite{jabali2014partial} with the implementation described in Section \ref{sec:implementation}, first with the E-cuts disabled, and then enabled. In Table \ref{Tab:Jabali_results}, the instances are grouped based on their number of nodes and vehicles. For example, 60\_2 refers to instances with $|N_0|=60$ nodes in which two vehicles must be used, i.e., $M=\{2\}$. For each variant of the algorithm and each group of instances, we report the number of instances solved to optimality, the average solving time (in seconds), the average optimality gap (in percentage), and the average number of nodes explored in the B\&B tree. For all experiments, we report average times and node counts based only on instances solved to optimality, and optimality gaps based only on instances that could not be solved within the time limit. Columns Time$^*$ and Nodes$^*$ of Table \ref{Tab:Jabali_results} report the average results of the DL-shaped algorithm with E-cuts on the instances that could be solved to optimality with both implementations. \rev{All the instances solved without E-cuts were also solved with E-cuts, so these two columns allow direct comparison with the Time and Nodes columns of the DL-shaped algorithm without E-cuts.}

\renewcommand{\arraystretch}{1.0} 
\begin{table}[H]
\caption{Impact of the E-cuts on performance, instances adapted from \cite{jabali2014partial}} \label{Tab:Jabali_results}
\centering
    \small
    \tighttable
    \begin{tabular*}{\textwidth}{@{\extracolsep{\fill}}crrrrrrrrrr}
\toprule
\multicolumn{1}{c}{\multirow{2}{*}{Set }} & \multicolumn{4}{c}{\textbf{DL-shaped without E-cuts}} & \multicolumn{6}{c}{\textbf{DL-shaped with E-cuts}} \\ \cmidrule(lr){2-5} \cmidrule(lr){6-11} 
 & Opt & \hspace{0.2cm} Gap & \hspace{0.2cm} Time & Nodes & Opt & Gap & Time & Nodes &
 Time$^*$ & Nodes$^*$ \\ \hline
60\_2 & 24 & 0.96 & 210 & 15,277 & 30 & 0.00 & 13  & 2,815 & 8  & 1,765 \\
70\_2 & 18 & 1.06 & 361 & 16,423 & 30 & 0.00 & 57  & 6,252 & 4  & 913  \\
80\_2 & 17 & 1.04 & 198 & 14,592 & 27 & 0.38 & 95  & 6,737 & 7  & 1,353 \\
50\_3 & 29 & 2.63 & 492 & 37,521 & 30 & 0.00 & 108 & 5,322 & 11 & 2,242 \\
60\_3 & 25 & 0.91 & 290 & 20,156 & 29 & 0.80 & 53  & 4,801 & 27 & 2,961 \\
70\_3 & 20 & 1.34 & 326 & 18,917 & 28 & 0.84 & 188 & 8,022 & 24 & 1,315 \\
40\_4 & 30 & 0.00 &  40 &  7,934 & 30 & 0.00 & 16  & 2,652 & 16 & 2,652 \\
50\_4 & 25 & 0.66 & 190 & 11,171 & 30 & 0.00 & 196 & 8,441 & 35 & 3,769 \\
60\_4 & 29 & 0.90 & 460 & 25,625 & 30 & 0.00 & 114 & 5,826 & 75 & 4,622 \\ \hline
All & 217 & 1.07 & 337 & 19,083 & 264 & 0.60 & 93 & 5,625 & 25 & 2,557 \\ \hline
\end{tabular*}
\end{table}

The results of Table \ref{Tab:Jabali_results} show that the E-cuts substantially improve the overall performance of our DL-shaped algorithm. Notably, their addition reduces the number of unsolved instances from 53 to 6. Furthermore, on the 217 instances that were solved by both versions of the algorithm, the E-cuts reduce the average solving time by a factor of more than 13, and the average number of explored nodes by a factor of more than 7. We highlight that these results are achieved based on the simple edge-set selection strategy described at the end of Section \ref{sec:E-cuts}, and that further refinement to this strategy might speed up computations even more. For the remainder of the experiments, the DL-shaped method is always implemented with the E-cuts enabled. 

\subsection{Comparison with state-of-the-art algorithms} \label{sec:experiments:existing}

In this section, we compare our DL-shaped algorithm with two integer L-shaped algorithms from the literature (LSG18, presented in \cite{louveaux2018exact} and HS25, presented in \cite{hoogendoorn2025evaluation}), and \rev{two BP\&C algorithms (FHM20, presented in \cite{florio2020new}, and FGHMV23, presented in \cite{florio2023recent}).} The results we report for each existing method are taken from the corresponding papers. Tables~\ref{tab:louveaux} and~\ref{tab:CVRP} present results for the instances of \cite{louveaux2018exact} and the CVRPLIB instances under the four variants defined in Section~\ref{subsec:VRPSDform}. Results from the literature are reported for the variants considered in each prior work.

\begin{table}[htbp]
    \caption{Comparison with existing methods, instances of \cite{louveaux2018exact}}
    \label{tab:louveaux}
    \centering
    \small
    \tighttable
    \setlength{\tabcolsep}{2pt}
    \begin{tabular*}{\textwidth}{@{\extracolsep{\fill}}ccrrrrrrrrrrrrrr}
    \toprule
    \multirow{2}{*}{Set} & \multirow{2}{*}{Variant} & \multicolumn{4}{c}{\textbf{LSG18}} &  & \multicolumn{4}{c}{\textbf{HS25}} &  & \multicolumn{4}{c}{\textbf{DL-shaped}} \\ \cmidrule{3-6} \cmidrule{8-11} \cmidrule{13-16} 
     &  &  & Time & Gap & Nodes &  & Opt & Time & Gap & Nodes &  & Opt & Time & Gap & Nodes \\ \hline
    E031\_09h & \multirow{4}{*}{ECC-FRC} & 7/8 & 205 & 2.33 & 57,594 &  & 8/8 & 8 & 0.00 & 2,036 &  & 8/8 & 7 & 0.00 & 583 \\
    E051\_05e &  & 7/8 & 639 & 0.13 & 50,526 &  & 8/8 & 29 & 0.00 & 2,677 &  & 8/8 & 7 & 0.00 & 985 \\
    E076\_07s &  & 7/8 & 3,426 & 1.03 & 162,910 &  & 8/8 & 242 & 0.00 & 5,283 &  & 8/8 & 186 & 0.00 & 5,159 \\
    E101\_08e &  & 6/8 & 1,580 & 1.20 & 77,235 &  & 6/8 & 98 & 0.68 & 1,501 &  & 7/8 & 442 & 0.74 & 16,595 \\ \hline
    E031\_09h & \multirow{4}{*}{ECC} & - & - & - & - &  & 8/8 & 8 & 0.00 & 2,096 &  & 8/8 & 8 & 0.00 & 711 \\
    E051\_05e &  & - & - & - & - &  & 8/8 & 30 & 0.00 & 2,813 &  & 8/8 & 6 & 0.00 & 851 \\
    E076\_07s &  & - & - & - & - &  & 8/8 & 164 & 0.00 & 4,308 &  & 8/8 & 97 & 0.00 & 3,549 \\
    E101\_08e &  & - & - & - & - &  & 6/8 & 77 & 0.71 & 1,282 &  & 7/8 & 417 & 0.94 & 16,110 \\ \hline
    E031\_09h & \multirow{4}{*}{FRC} & - & - & - & - &  & 3/8 & 42 & 5.90 & 9,032 &  & 8/8 & 106 & 0.00 & 11,874 \\
    E051\_05e &  & - & - & - & - &  & 0/8 & - & 3.47 & - &  & 5/8 & 899 & 1.44 & 25,621 \\
    E076\_07s &  & - & - & - & - &  & 0/8 & - & 5.38 & - &  & 4/8 & 70 & 1.99 & 4,045 \\
    E101\_08e &  & - & - & - & - &  & 0/8 & - & 2.81 & - &  & 5/8 & 520 & 1.19 & 14,335 \\ \hline
    E031\_09h & \multirow{4}{*}{Basic} & - & - & - & - &  & 1/8 & 752 & 7.77 & 53,625 &  & 4/8 & 48 & 3.65 & 5,035 \\
    E051\_05e &  & - & - & - & - &  & 0/8 & - & 4.67 & - &  & 3/8 & 38 & 2.71 & 5,349 \\
    E076\_07s &  & - & - & - & - &  & 0/8 & - & 5.66 & - &  & 3/8 & 42 & 2.74 & 1,942 \\
    E101\_08e &  & - & - & - & - &  & 0/8 & - & 3.78 & - &  & 3/8 & 1,083 & 2.02 & 8,286 \\ \hline
    All &  & \hspace{-1cm}27/32 \hspace{-0.22cm} &  &  &  &  & \hspace{-1cm}64/128\hspace{-0.22cm} &  &  &  &  & \hspace{-1cm}97/128\hspace{-0.22cm} &  &  &  \\ \hline
    \end{tabular*}
\end{table}

The results in Table~\ref{tab:louveaux} indicate that our DL-shaped method achieves state-of-the-art results for the instances of \cite{louveaux2018exact}. For both the ECC-FRC and ECC variants, we solve one of the two open instances from this benchmark set. The DL-shaped algorithm also consistently explores a smaller number of nodes and achieves better computing times than LSG18 and HS25. The advantage of the DL-shaped method is most significant on the challenging FRC and basic variants, where we solve a total of 35 instances to optimality, compared to only 4 for HS25. In these variants, the absence of ECCs allows for the construction of very long routes, which makes the set of feasible routes extremely large. As a consequence, valid inequalities that remain active over many routes become critical for efficiently approximating the recourse function. \rev{This is precisely the structural advantage of the DL-shaped cuts, whose activation depends only on the internal subpath structure of a solution, not on how the surrounding routes connect to the depot. By contrast, the partial-route-based inequalities implemented in HS25 use depot-incident edges in their activation functions, restricting their activation to routes whose first and last visited customers match a prescribed configuration. This broader activation of the DL-shaped cuts is the same mechanism that underlies Theorem~\ref{thm:ecut-prs} and the algorithmic benefits of the E-cuts illustrated on the wheel-graph instances of Section~\ref{sec:E-cuts:wheel}.} We do not report BP\&C results for the instances of \cite{louveaux2018exact} since, as discussed in \cite{florio2023recent}, BP\&C methods are not competitive on this group of instances due to their large customer-to-vehicle ratio.

\begin{table}[htbp]
    \caption{Comparison with existing methods, CVRPLIB instances \citep{uchoa2017new}}
    \label{tab:CVRP}
    \centering
    \small
    \tighttable
    \begin{tabular*}{\textwidth}{@{\extracolsep{\fill}}ccrrrrrrrrrrrr}
    \toprule
    \multicolumn{1}{c}{\multirow{2}{*}{Set}} & \multicolumn{1}{c}{\multirow{2}{*}{Variant}} & \multicolumn{3}{c}{\rev{\textbf{BP\&C}$^\dagger$}} &  & \multicolumn{3}{c}{\textbf{HS25}} &  & \multicolumn{3}{c}{\textbf{DL-shaped}} \\
    \cmidrule{3-5} \cmidrule{7-9} \cmidrule{11-13}
    & & \rev{Opt} & \rev{Time} & \rev{Gap} & & Opt & Time & Gap & & Opt & Time & Gap \\ \hline
    A & \multirow{5}{*}{ECC-FRC} & - & - & - & & 1/27 & 1,296 & - & & 4/27 & 759 & 10.14 \\
    B & & - & - & - & & - & - & - & & 2/23 & 1,910 & 9.18 \\
    E & & - & - & - & & 6/13 & 346 & - & & 5/13 & 19 & 8.92 \\
    F & & - & - & - & & - & - & - & & 1/3 & 5 & 1.78 \\
    P & & - & - & - & & 7/24 & 61 & - & & 8/24 & 412 & 10.40 \\ \hline
    A & \multirow{5}{*}{ECC} & \rev{20/27} & \rev{1,425} & \rev{0.31$^{(5/7)}$} & & 1/27 & 1,498 & - & & 6/27 & 923 & 8.27 \\
    B & & - & - & - & & - & - & - & & 2/23 & 1,266 & 5.95 \\
    E & & \rev{3/8} & \rev{480} & \rev{0.35$^{(4/5)}$} & & 6/13 & 158 & - & & 6/13 & 174 & 8.15 \\
    F & & - & - & - & & - & - & - & & 1/3 & 6 & 2.07 \\
    P & & \rev{20/23} & \rev{1,884} & \rev{7.07$^{(1/3)}$} & & 7/24 & 11 & - & & 8/24 & 333 & 7.46 \\ \hline
    A & \multirow{5}{*}{FRC} & - & - & - & & 0/27 & - & 20.39 & & 0/27 & - & 36.98 \\
    B & & - & - & - & & - & - & - & & 0/23 & - & 38.07 \\
    E & & - & - & - & & 0/13 & - & 17.62 & & 4/13 & 276 & 21.53 \\
    F & & - & - & - & & - & - & - & & 0/3 & - & 17.12 \\
    P & & - & - & - & & 4/24 & 399 & 20.35 & & 6/24 & 218 & 19.63 \\ \hline
    A & \multirow{5}{*}{basic} & - & - & - & & 0/27 & - & 26.49 & & 0/27 & - & 58.58 \\
    B & & - & - & - & & - & - & - & & 0/23 & - & 67.08 \\
    E & & - & - & - & & 0/13 & - & 29.46 & & 0/13 & - & 29.19 \\
    F & & - & - & - & & - & - & - & & 0/3 & - & 24.99 \\
    P & & - & - & - & & 1/24 & 1,055 & 24.59 & & 3/24 & 75 & 41.00 \\ \hline
    All & & \rev{43/58} & & & & 33/256 & & & & 56/360 & & \\ \hline
    \end{tabular*}
    \par\smallskip
    \begin{minipage}{\textwidth}
    \raggedright
    \small \rev{$^\dagger$ best result across FHM20 \citep{florio2020new} and FGHMV23 \citep{florio2023recent}. Gap superscripts of the form $(a/b)$ indicate that the average gap is computed over $a$ instances among the $b$ that could not be solved to optimality (the remaining $b-a$ have no reported lower bound)}
    \end{minipage}
\end{table}

Unlike the instances of \cite{louveaux2018exact}, the CVRPLIB instances are best-suited for BP\&C algorithms, with customers-per-route ratios rarely exceeding 10 when ECCs are imposed. \rev{FGHMV23 is an elementary BP\&C method that improves upon the dominance-based BP\&C algorithm FHM20 and represents the current state-of-the-art among BP\&C methods. Since FHM20 and FGHMV23 test overlapping but distinct subsets of these instances, we report for each instance the best result obtained by either method, and refer to this combined reference as BP\&C in Table~\ref{tab:CVRP}. The results show that BP\&C methods are dominant for the CVRPLIB instances in the ECC variant, jointly solving 43 of the 58 instances on which they were tested.} Among B\&C methods, our DL-shaped algorithm provides better results than HS25 on sets A and P, whereas HS25 appears to perform slightly better on set E in the ECC-FRC and ECC variants. Since the results of HS25 were obtained using an Intel Xeon W-2123 3.6 GHz processor, which is faster than the processor we used to test our DL-shaped algorithm, replication on the same computing infrastructure would be needed to properly determine the best-performing algorithm for these instances. What is unambiguous from these results is that the DL-shaped algorithm is dominant on the FRC and basic variants, where it solves 13 instances to optimality, compared to 5 for HS25. Overall, our algorithm solves 17 more instances than HS25 among the 256 CVRPLIB instances on which both methods were tested, and thus clearly achieves state-of-the-art results among B\&C methods.

\subsection{Results on new instances} \label{sec:experiments:characteristics}
In Table \ref{Tab:parada_results}, we present results for the new benchmark set adapted from \cite{parada2024disaggregated}. The instances are organized by the number of routes $M=\{\bar{m}\}$ and the number of customers $n$. Two key insights emerge from these experiments. First, for a fixed number of customers, increasing the number of vehicles leads to instances that are significantly more challenging for the DL-shaped method. While the instances with 30 customers and two vehicles can be solved in 18 seconds on average, only 2 out of 30 instances can be solved to optimality when the number of routes is increased to seven. Second, the DL-shaped method demonstrates excellent scalability with respect to the customer-to-vehicle ratio, solving 27 out of the 30 instances with 120 customers and two vehicles. \rev{This customer-to-vehicle ratio of 60 substantially exceeds what has been previously addressed in the exact VRPSD literature.} Indeed, the solved instances with the largest customer-to-vehicle ratio comprise 80 customers and two vehicles \rev{for the B\&C algorithms of \cite{hoogendoorn2023improved, hoogendoorn2025evaluation}}, and only 21 customers and two vehicles \rev{for the BP\&C methods of \cite{florio2020new, florio2023recent}}. These results confirm that our DL-shaped algorithm excels at solving instances with a small number of long routes\rev{, complementing the BP\&C methods, which dominate in the many-short-routes regime}. 

\begin{table}[htbp]
\caption{Results for the DL-shaped method, instances adapted from \cite{parada2024disaggregated}} \label{Tab:parada_results}
\centering
\small       
\tighttable
\setlength{\tabcolsep}{2pt}
\begin{tabular*}{\textwidth}{@{\extracolsep{\fill}}crrrrrrrrrrrrrr}
\toprule
\multicolumn{1}{c}{\multirow{2}{*}{ \diagbox{$n$}{$\bar{m}$} }} & \multicolumn{2}{c}{2} & \multicolumn{2}{c}{3} & \multicolumn{2}{c}{4} & \multicolumn{2}{c}{5} & \multicolumn{2}{c}{6} & \multicolumn{2}{c}{7} & \multicolumn{2}{c}{All} \\ \cmidrule(lr){2-3}\cmidrule(lr){4-5} \cmidrule(lr){6-7} \cmidrule(lr){8-9} \cmidrule(lr){10-11} \cmidrule(lr){12-13} \cmidrule(lr){14-15}
\multicolumn{1}{c}{} & Opt & Time & Opt & Time & Opt & Time & Opt & Time & Opt & Time & Opt & Time & Opt & Time \\ \hline
20  & 30  & 2   & 30  & 4   & 30  & 23  & 27 & 200    & 28 & 454  & 30 & 172  & 175 & 138 \\
30  & 30  & 18  & 29  & 135 & 27  & 314 & 19 & 441    & 7  & 1,673 & 2  & 2,797 & 114 & 339 \\
40  & 30  & 7   & 28  & 77  & 22  & 235 & 11 & 593    & 3  & 1,510 & 0  & -    & 94  & 197 \\
50  & 30  & 4   & 26  & 320 & 16  & 319 & 11 & 848    & 0  & -    & 0  & -    & 83  & 276 \\
60  & 30  & 48  & 27  & 325 & 16  & 582 & 5  & 1,112  & 0  & -    & 0  & -    & 78  & 322 \\
70  & 30  & 26  & 20  & 230 & 15  & 543 & 1  & 72     & 1  & 3,535 & 0  & -    & 67  & 256 \\
80  & 28  & 28  & 21  & 239 & 11  & 386 & 2  & 2,151  & 0  & -    & 0  & -    & 62  & 231 \\
90  & 29  & 103 & 20  & 416 & 10  & 377 & 3  & 508    & 0  & -    & 0  & -    & 62  & 268 \\
100 & 30  & 159 & 18  & 181 & 9   & 199 & 2  & 1,507  & 0  & -    & 0  & -    & 59  & 218 \\
110 & 28  & 108 & 18  & 389 & 8   & 627 & 1  & 565    & 0  & -    & 0  & -    & 55  & 284 \\
120 & 27  & 120 & 17  & 102 & 9   & 702 & 1  & 3,577  & 0  & -    & 0  & -    & 54  & 275 \\ \hline
All & 322 & 56  & 254 & 210 & 173 & 335 & 83 & 581    & 39 & 833  & 32 & 336  & 903 & 244\\ \hline
\end{tabular*}
\end{table}

\rev{
\section{Asymptotic analysis}\label{sec:asymptotics}
In this section, we complement our algorithmic analysis of the VRPSD with an asymptotic study of its optimal value. We consider customers drawn as i.i.d.\ (location, demand) pairs with locations in a compact region of the Euclidean plane. Building on classical asymptotic results of \cite{haimovich1985bounds} for the capacitated vehicle routing problem (CVRP) and the split-delivery vehicle routing problem (SDVRP), as well as on the Beardwood–Halton–Hammersley (BHH) theorem of \cite{beardwood1959shortest} for the TSP, we establish three results.
\begin{enumerate}
    \item Under both the DTD and OR policies, the asymptotic per-customer cost of the basic-VRPSD matches the SDVRP asymptotic rate (Theorem~\ref{thm:asymptotic}).
    \item Under the ECCs, this rate is multiplied by a factor in $[\alpha, \alpha+1]$, where $\alpha \in [1, 2)$ is a bin-packing constant determined by the vehicle capacity and the expected customer demands (Theorem~\ref{thm:ECC_rate}). Both endpoints of this band are approached arbitrarily closely by suitable distributions (Proposition~\ref{prop:ECC_tightness}).
    \item As a corollary, we resolve, in the asymptotic regime under study, an open question of \cite{hoogendoorn2025evaluation}: imposing the ECCs multiplies the optimal value of the basic-VRPSD by at most 3, and this bound is tight (Corollary~\ref{cor:ECC_worst_case}).
\end{enumerate}

Among the variants of the VRPSD discussed in Section~\ref{subsec:VRPSDform}, we do not consider those imposing the FRC, as fixing the number of routes a priori is not naturally compatible with the i.i.d.\ asymptotic regime under study.

The remainder of this section is structured as follows. Section~\ref{sec:asymptotics:setting} formalizes the setup. Section~\ref{sec:asymptotics:sdvrp} reviews the CVRP and SDVRP rates that anchor our analysis. Section~\ref{sec:asymptotics:accounting} establishes preliminary bounds on the DTD recourse function for i.i.d.\ demands. Sections~\ref{sec:asymptotics:main} and~\ref{sec:asymptotics:ECC} then present the asymptotic analyses of the basic-VRPSD and the ECC-VRPSD, respectively.

\subsection{Setting and assumptions}\label{sec:asymptotics:setting}
We consider the Euclidean space $\mathbb{R}^2$ equipped with the natural Euclidean distance, where the depot location $X_0$ is the origin. Let $(X_i, \xi_i)_{i \geq 1}$ be a sequence of i.i.d.\ random variables on $\mathcal{R} \times \mathbb{R}_{\geq 0}$, where $\mathcal{R} \subset \mathbb{R}^2$ is a compact set, and the demand $\xi_i$ of customer $i$ is independent of its location~$X_i$. Denote by $Q\in \mathbb{R}_{>0}$ the vehicle capacity, $\mu := \mathbb{E}[\xi_i]\in \mathbb{R}_{>0}$ the expected customer demand, and $\bar{d} := \mathbb{E}[\|X_i\|]\in \mathbb{R}_{>0}$ the expected distance from a customer to the depot. Let $L(p) := \sum_{j=0}^{t} \|X_{i_j} - X_{i_{j+1}}\|$ be the length of the route $(i_0{=}0, p, 0{=}i_{t+1})$ associated with the path $p = (i_1, \ldots, i_t)$. We assume recourse costs to be purely distance-based ($b^F = b^P = 0$); in particular, $c^F_i=2\|X_i\|$. 

\subsection{Asymptotic rates for the CVRP and the SDVRP}\label{sec:asymptotics:sdvrp}

In the CVRP, a fleet of identical vehicles must serve a set of customers with known demands, each assigned to a single vehicle, so as to minimize the total traveled distance. In a seminal paper, \cite{haimovich1985bounds} established asymptotic bounds for the CVRP with identical demands and unlimited fleet. Assume each demand equals $\mu$, with $\mu \leq Q$, so that a vehicle can serve at most $\kappa := \lfloor Q/\mu \rfloor$ customers. This integer rounding introduces a gap between the vehicle's nominal capacity $Q$ and its usable capacity $\mu\kappa$. Define the \emph{bin-packing constant}:
\begin{equation}\label{eq:alpha_def}
    \alpha := \frac{Q}{\mu\kappa} \in [1, 2),
\end{equation}
as the ratio of capacity to maximum usable capacity when packing items of size $\mu$ into bins of size $Q$. We have $\alpha = 1$ when $Q/\mu \in \mathbb{Z}$ (capacity is fully usable), and $\alpha < (\kappa + 1)/\kappa \leq 2$ otherwise. For the first $n$ points of the location sequence $(X_i)_{i \geq 1}$, denote by $V_n^{\text{CVRP}}(X)$ the optimal value of this CVRP. Writing the customer-count capacity of each vehicle as $\kappa = Q/(\mu\alpha)$, we obtain the following asymptotic rate from \cite{haimovich1985bounds}.
\begin{theorem}[\cite{haimovich1985bounds} Theorem~4(i)]\label{thm:haimovich_cvrp}
    Under the assumptions of Section~\ref{sec:asymptotics:setting} with $\mu \leq Q$,
    \begin{equation}\label{eq:unsplit_rate}
        \lim_{n \to \infty} \frac{V_n^{\text{CVRP}}(X)}{n} = \alpha \cdot \frac{2\mu\bar d}{Q} \qquad \text{a.s.}
    \end{equation}
\end{theorem}

In the same paper, \cite{haimovich1985bounds} extended their results to customers with heterogeneous demands, provided that each customer's demand may be split across multiple vehicles. This split-delivery variant of the CVRP was later formalized as the SDVRP by~\cite{dror1989savings} and has since been extensively studied~\cite{archetti2012vehicle}. Let $N = \{1, \dots, n\}$ be the first $n$ customers of the sequence $(X_i, \xi_i)_{i \geq 1}$. A solution of the SDVRP for these realized locations and demands consists of a set $\{(0, p^k, 0)\}_{k=1}^K$ of $K \in \mathbb{N}$ routes covering $N$, i.e., $\bigcup_{k=1}^K N(p^k) = N$, together with delivery quantities $\{y_i^k \geq 0 : i \in N(p^k)\}_{k=1}^K$ that jointly fulfill the demands and respect each vehicle's capacity:
\begingroup
\allowdisplaybreaks
\begin{alignat}{2}
    & \sum_{k : i \in N(p^k)} y_i^k = \xi_i, \qquad & \forall i \in N, \label{SDVRP:demands} \\
    & \sum_{i \in N(p^k)} y_i^k \leq Q, \qquad & \forall k \in \{1, \dots, K\}. \label{SDVRP:capacity}
\end{alignat}
\endgroup
The objective is to minimize the total traveled distance $\sum_{k=1}^K L(p^k)$. We denote by $V_n^{\text{SDVRP}}(X, \xi)$ the optimal value of this SDVRP. 
\begin{theorem}[\cite{haimovich1985bounds} Section~7, extension of Theorem~4(i)]\label{thm:haimovich_thm4}
    Under the assumptions of Section~\ref{sec:asymptotics:setting},
    \begin{equation}\label{eq:SDVRP_rate}
        \lim_{n \to \infty} \frac{V_n^{\text{SDVRP}}(X, \xi)}{n} = \frac{2\mu\bar d}{Q} \qquad \text{a.s.}
    \end{equation}
\end{theorem}

The SDVRP rate $2\mu\bar d/Q$ is the expected round-trip distance $2\bar d$ between the depot and a customer, multiplied by the ratio $\mu/Q$ of expected demand to vehicle capacity. The CVRP rate multiplies the SDVRP rate by the bin-packing constant $\alpha$, reflecting the capacity waste from integer rounding when demand cannot be split across multiple vehicles.

\subsection{Preliminary lemmas: DTD recourse function for i.i.d.\ demands}\label{sec:asymptotics:accounting}
Before establishing our main results, we develop two preliminary lemmas on the DTD recourse function for i.i.d.\ demands. These will support the upper-bound arguments of Sections \ref{sec:asymptotics:main} and \ref{sec:asymptotics:ECC}.

Consider the route formed by path $p = (i_1, \ldots, i_t)$. Under the DTD policy, a failure at position $j$ occurs whenever the cumulative demand $\sum_{k=1}^{j} \xi_{i_k}$ crosses a multiple of the vehicle capacity $Q$. When demands are i.i.d.\ and independent of locations, the expected number of failures at position $j$ is independent of the order of visit of the customers; we denote this quantity by the constant:
\begin{equation}\label{eq:phi_def}
    \phi_j := \sum_{l \geq 1} \mathbb{P}\!\left[\sum_{k=1}^{j-1} \xi_k \leq lQ < \sum_{k=1}^{j} \xi_k\right].
\end{equation}

\begin{lemma}\label{lem:accounting} Under the assumptions of Section~\ref{sec:asymptotics:setting}, the expected number of failures $\sum_{k=1}^j \phi_k$ for the first $j$ customers of a route executed under the DTD policy satisfies:
\begin{equation}\label{ineq:accounting}
    j\mu/Q - 1 \leq \sum_{k=1}^j \phi_k \leq j\mu/Q.
\end{equation}
\end{lemma}
\begin{proof}
    Let $R_j$ denote the total number of failures over the first $j \leq t$ customers of a route $(0, i_1, \ldots, i_t,0)$ executed under the DTD policy. The vehicle loads $Q$ units upon departure from the depot and an additional $Q$ units at each of these $R_j$ restocking trips, for a cumulative load of $(R_j+1)Q$. Since every loaded unit is either delivered or remains as the residual capacity $q_{j+1} \in [0, Q]$ after serving the $j$-th customer, $(R_j + 1) Q = \sum_{k=1}^j \xi_{i_k} + q_{j+1}$. Taking expectations gives $(\mathbb{E}[R_j] + 1) Q = j\mu + \mathbb{E}[q_{j+1}] \in [j\mu, j\mu + Q]$. Since $\mathbb{E}[R_j]=\sum_{k=1}^j \phi_k$ holds by definition, reorganizing yields \eqref{ineq:accounting}.
\end{proof}

In this section, we thus write the DTD expected recourse \eqref{rec:DTD} as:
\begin{equation}\label{eq:recourse_DTD}
    \bar{\mathcal{Q}}^{\text{DTD}}_{p} = \sum_{j=1}^t 2 \phi_j \, \| X_{i_j}\|.
\end{equation}

We now use the failure-count bound of Lemma~\ref{lem:accounting} to bound the DTD expected recourse of a route as a function of its length.

\begin{lemma}\label{lem:recourse_bound}
    Under the assumptions of Section~\ref{sec:asymptotics:setting}, the DTD expected recourse of the route $(0, p, 0)$ formed by the path $p = (i_1, \dots, i_t)$ satisfies:
    \begin{equation}\label{ineq:recourse_bound}
        \bar{\mathcal{Q}}^{\text{DTD}}_p \leq L(p) \cdot \frac{t\mu}{Q}.
    \end{equation}
\end{lemma}

\begin{proof}
    By the triangle inequality, the length $L(p)$ of $(0,p,0)$ is at least $2\|X_{i_j}\|$ for each visited customer $j \in \{1, \dots, t\}$.  Combining with~\eqref{eq:recourse_DTD} and Lemma~\ref{lem:accounting} yields the result:
    \begin{equation*}
\bar{\mathcal{Q}}^{\text{DTD}}_p = \sum_{j=1}^t 2\phi_j \|X_{i_j}\| \leq L(p) \sum_{j=1}^t \phi_j \leq L(p) \cdot \frac{t\mu}{Q}.
\end{equation*}\end{proof}

\subsection{Asymptotic rate of the basic-VRPSD}\label{sec:asymptotics:main}

In this section, we show that the per-customer cost of the basic-VRPSD converges almost surely to $2\mu\bar d/Q$, matching the SDVRP rate. The proof constructs matching lower and upper bounds. In Lemma~\ref{lem:asymptotic_lb}, we show that the realized trajectory of an optimal basic-VRPSD solution, after recourse, forms a feasible SDVRP solution, yielding a lower bound via the SDVRP rate. The upper bound is obtained in Lemma~\ref{lem:asymptotic_ub} by constructing suboptimal basic-VRPSD solutions from optimal TSP tours; their per-customer first-stage cost is asymptotically negligible by the BHH theorem, and we show that their expected recourse cost under the DTD policy asymptotically matches the SDVRP rate.

We introduce notation, then state the result. Let $N = \{1,\dots,n\}$ be the first $n$ customers of the sequence $(X_i,\xi_i)_{i\ge1}$, and let $\mathcal P_n$ denote the set of Hamiltonian paths on $N$. Let $C^{\pi}(p, \xi)$ denote the realized recourse cost for a path $p \in \mathcal{P}_n$, a policy $\pi \in \{\text{DTD}, \text{OR}\}$, and a demand realization $\xi = (\xi_1, \dots, \xi_n)$. Under policy $\pi$, the realized travel distance when executing route $(0, p, 0)$ is $D^{\pi}(p, \xi) := L(p) + C^{\pi}(p, \xi)$, and the expected cost of the route, conditional on the locations, is $V^{\pi}(p; X) := \mathbb{E}_{\xi}\!\left[D^{\pi}(p, \xi) \mid X\right] = L(p) + \bar{\mathcal{Q}}^{\pi}_p$, with the expected recourse cost $\bar{\mathcal{Q}}^{\pi}_p = \mathbb{E}_\xi[C^{\pi}(p, \xi) \mid X]$ as defined in Section~\ref{subsec:ExpRecCost}. We denote by $V^{\pi}_n(X)$ the optimal value of the basic-VRPSD over these $n$ customers under policy $\pi$.

\begin{theorem}\label{thm:asymptotic}
    Under the assumptions of Section~\ref{sec:asymptotics:setting}, each policy $\pi \in \{\text{DTD}, \text{OR}\}$ satisfies:
    \begin{equation}\label{eq:asymptotic_limit}
        \lim_{n \to \infty} \frac{V^{\pi}_n(X)}{n} = \frac{2\mu\bar{d}}{Q} 
        \qquad \text{a.s.}
    \end{equation}
\end{theorem}

\begin{proof} 
The result follows from the sandwich:
\begin{equation}
     \frac{2\mu\bar d}{Q} \leq \liminf_{n \to \infty} \frac{V^{\text{OR}}_n(X)}{n} \leq \limsup_{n \to \infty} \frac{V^{\text{DTD}}_n(X)}{n} \leq \frac{2\mu\bar d}{Q}  \quad \text{a.s.}
\end{equation}
The first and third inequalities are established in Lemmas \ref{lem:asymptotic_lb} and \ref{lem:asymptotic_ub}, respectively. The second inequality holds since $V^{\text{OR}}_n(X) \leq V^{\text{DTD}}_n(X)$ for any realized locations $(X_{i})_{i=1}^n$, as the feasible first-stage solutions are the same regardless of the recourse policy and $\bar{\mathcal{Q}}_p^{\text{OR}} \leq \bar{\mathcal{Q}}_p^{\text{DTD}}$ holds for every path $p$ by Remark~\ref{rem:DTD_restriction_of_OR}.
\end{proof}

\begin{lemma}\label{lem:asymptotic_lb}
Under the assumptions of Section~\ref{sec:asymptotics:setting},
\begin{equation}\label{eq:lower_bound}
    \liminf_{n \to \infty} \frac{V^{\text{OR}}_n(X)}{n} \geq \frac{2\mu\bar d}{Q} \quad \text{a.s.}
\end{equation}
\end{lemma}
\begin{proof}
By \cite[Theorem~2]{yang2000stochastic}, the basic-VRPSD under the OR policy always admits an optimal first-stage solution consisting of a single route. For any $n\geq 1$, we can thus write $V^{\text{OR}}_n(X) = V^{\text{OR}}(p^*_n; X)$ for some path $p^*_n \in \mathcal{P}_n$. Fix a demand realization $\xi = (\xi_i)_{i=1}^n$. Executing the OR policy on $(0, p^*_n, 0)$ produces a realized trajectory that decomposes into depot-to-depot trips $\{(0, p^k, 0)\}_{k=1}^K$, for some $K \in \mathbb{N}$. Setting $y_i^k$ to the amount delivered to customer $i \in N(p^k)$ during trip $k$ yields delivery quantities $\{y_i^k \geq 0 : i \in N(p^k)\}_{k=1}^K$ that fulfill the realized demand of every customer and respect the vehicle capacity on each trip, hence jointly satisfy~\eqref{SDVRP:demands}--\eqref{SDVRP:capacity}. Since the realized VRPSD trajectory thus constitutes a feasible SDVRP solution, its total distance $D^{\text{OR}}(p^*_n, \xi)$ is no less than the SDVRP optimum $V^{\text{SDVRP}}_n(X, \xi)$. From there, applying the lower bound from \cite[Section 7, Lemma~1]{haimovich1985bounds} on the SDVRP optimum yields:
\begin{equation}\label{eq:radial}
D^{\text{OR}}(p^*_n, \xi) \geq V^{\text{SDVRP}}_n(X, \xi) \geq \frac{2}{Q}\sum_{i=1}^n \xi_i \|X_i\|.
\end{equation}
Since~\eqref{eq:radial} holds for every realization~$\xi$, taking 
conditional expectation given~$X$ and using $\mathbb{E}[\xi_i \mid X] = \mu$, which holds by independence of locations and demands, yields, for every $n \geq 1$,
\begin{equation}\label{eq:lb_V^OR_n}
    V^{\text{OR}}_n(X) \geq \frac{2\mu}{Q}\sum_{i=1}^n \|X_i\|
    \qquad \text{a.s.}
\end{equation}

By the strong law of large numbers, $\frac{1}{n}\sum_{i=1}^n \|X_i\| \to \bar{d}$ a.s., where $\bar{d} < \infty$ by boundedness of $\mathcal{R}$. Dividing \eqref{eq:lb_V^OR_n} by $n$ and passing to the $\liminf$ therefore gives \eqref{eq:lower_bound}.
\end{proof}

\begin{lemma}\label{lem:asymptotic_ub}
Under the assumptions of Section~\ref{sec:asymptotics:setting},
\begin{equation}\label{eq:upper_bound}
    \limsup_{n \to \infty} \frac{V^{\text{DTD}}_n(X)}{n} \leq \frac{2\mu\bar d}{Q} \quad \text{a.s.}
\end{equation}
\end{lemma}
\begin{proof}
For each $n \geq 1$, let $p^{\text{TSP}}_n = (i_1, \dots, i_n) \in 
\mathcal{P}_n$ be a Hamiltonian path on $N$ whose associated cycle 
$(0, p^{\text{TSP}}_n, 0)$ is a shortest TSP tour through $N \cup \{0\}$. 
Since the single-route solution $\{(0, p^{\text{TSP}}_n, 0)\}$ is feasible 
for the basic-VRPSD, $V^{\text{DTD}}_n(X) \leq V^{\text{DTD}}(p^{\text{TSP}}_n; X) = L(p^{\text{TSP}}_n) + \bar{\mathcal{Q}}^{\text{DTD}}_{p^{\text{TSP}}_n}$, 
so it suffices to show that:
\begin{equation}\label{eq:upper_goal}
    \lim_{n \to \infty} \frac{ L(p^{\text{TSP}}_n) + \bar{\mathcal{Q}}^{\text{DTD}}_{p^{\text{TSP}}_n} }{n} 
     =  \frac{2\mu\bar d}{Q} \quad \text{a.s.}
\end{equation}
By \eqref{eq:recourse_DTD}, the expected recourse cost can be expressed as $\bar{\mathcal{Q}}^{\text{DTD}}_{p^{\text{TSP}}_n} = \sum_{j=1}^n 2\phi_j \|X_{i_j}\|$. Writing $\phi_j = \mu/Q + (\phi_j - \mu/Q)$, we decompose $\bar{\mathcal{Q}}^{\text{DTD}}_{p^{\text{TSP}}_n} = A_n + B_n$, where:
\begin{equation}\label{eq:DTD_TSP_decomp}
    A_n := \frac{2\mu}{Q} \sum_{j=1}^n \|X_{i_j}\|
    \qquad \text{and} \qquad 
    B_n := 2 \sum_{j=1}^n (\phi_j - \mu/Q) \|X_{i_j}\|.
\end{equation}

We now study the sequences of random variables $(A_n/n)_{n \geq 1}$ and $(B_n/n)_{n \geq 1}$.

\emph{Leading term.} Since $(X_{i_1}, \dots, X_{i_n})$ is a permutation 
of $(X_1, \dots, X_n)$, $A_n$ is invariant under the choice of tour 
ordering, and the strong law of large numbers applied to the i.i.d.\ sequence $(\|X_i\|)_{i \geq 1}$ 
gives
\[
    \frac{A_n}{n} = \frac{2\mu}{Q} \cdot \frac{1}{n} \sum_{k=1}^n \|X_k\| 
    \longrightarrow \frac{2\mu\bar d}{Q} \quad \text{a.s.}
\]

\emph{Remainder.} Fix $n \geq 1$. Summation by parts on the sequences $(\phi_j - \mu/Q)_{j \geq 1}$ 
and $(\|X_{i_j}\|)_{j \geq 1}$, the former with partial sums $S_j := \sum_{k=1}^j \phi_k - j\mu/Q$, gives:
\[
    \sum_{j=1}^n (\phi_j - \mu/Q) \|X_{i_j}\| 
    = S_n \|X_{i_n}\| - \sum_{j=1}^{n-1} S_j \bigl(\|X_{i_{j+1}}\| - \|X_{i_j}\|\bigr).
\]
By Lemma~\ref{lem:accounting}, $|S_j| \leq 1$ uniformly in $j$, so the triangle inequality gives:
\[
    |B_n| = \Big| 2 S_n \|X_{i_n}\| - 2 \sum_{j=1}^{n-1} S_j \bigl(\|X_{i_{j+1}}\| - \|X_{i_j}\|\bigr) \Big| \leq 2\|X_{i_n}\| + 2\sum_{j=1}^{n-1} \|X_{i_{j+1}} - X_{i_j}\|.
\]
Since $X_{i_n} \in \mathcal{R}$ and $\sum_{j=1}^{n-1} \|X_{i_{j+1}} - X_{i_j}\|$ is bounded by the length $L(p^{\text{TSP}}_n)$ of the TSP tour $(0, p^{\text{TSP}}_n, 0)$, we obtain $|B_n| \leq 2\sup_{x \in \mathcal{R}}\|x\| + 2 L(p^{\text{TSP}}_n)$. The first term is finite by boundedness of $\mathcal{R}$, and $L(p^{\text{TSP}}_n) = O(\sqrt{n})$ a.s.\ by the BHH theorem~\cite{beardwood1959shortest}, so $L(p^{\text{TSP}}_n)/n \to 0$ a.s. We conclude that $B_n/n \to 0$ a.s.

\emph{Conclusion.} Combining $A_n/n \to 2\mu\bar d/Q$ a.s.\ and $B_n/n \to 0$ a.s.\ gives $\bar{\mathcal{Q}}^{\text{DTD}}_{p^{\text{TSP}}_n}/n \to 2\mu\bar d/Q$ a.s.\ Together with $L(p^{\text{TSP}}_n)/n \to 0$ a.s., this establishes~\eqref{eq:upper_goal}.
\end{proof}

\subsection{Asymptotic rate of the VRPSD with expected capacity constraints}\label{sec:asymptotics:ECC}

Under the ECCs, every route $(0,p,0)$ can serve at most $\kappa = \lfloor Q/\mu \rfloor$ customers. For feasibility, assume that $\mu \leq Q$, so that each vehicle can serve at least one customer, and denote by $V^{E, \pi}_n(X)$ the optimal value of the ECC-VRPSD for the first $n$ customers of the sequence $(X_i, \xi_i)_{i \geq 1}$ under policy $\pi \in \{\text{DTD}, \text{OR}\}$. In Theorem~\ref{thm:ECC_rate}, we show that the per-customer cost of the ECC-VRPSD lies asymptotically between $\alpha$ and $\alpha + 1$ times the SDVRP rate $2\mu\bar{d}/Q$. The proof parallels that of Theorem~\ref{thm:asymptotic} in structure, but invokes the CVRP rate of Theorem~\ref{thm:haimovich_cvrp} in place of the SDVRP rate for the lower bound, and constructs upper-bounding ECC-feasible solutions from optimal CVRP solutions rather than from TSP tours.

\begin{theorem}\label{thm:ECC_rate}
    Under the assumptions of Section~\ref{sec:asymptotics:setting} with $\mu \leq Q$, each policy $\pi \in \{\text{DTD}, \text{OR}\}$ satisfies:
    \begin{equation}\label{eq:ECC_rate}
        \alpha \cdot \frac{2\mu\bar d}{Q} \leq \liminf_{n \to \infty} \frac{V^{E, \pi}_n(X)}{n} \leq \limsup_{n \to \infty} \frac{V^{E, \pi}_n(X)}{n} \leq (\alpha + 1) \cdot \frac{2\mu\bar d}{Q} \quad \text{a.s.}
    \end{equation}
\end{theorem}

\begin{proof}
    The lower bound uses only non-negativity of recourse costs and is thus policy-independent. For the upper bound, it suffices to establish the result under DTD, since the feasible first-stage solutions are the same regardless of the recourse policy and $\bar{\mathcal{Q}}_p^{\text{OR}} \leq \bar{\mathcal{Q}}_p^{\text{DTD}}$ holds for every path $p$ by Remark~\ref{rem:DTD_restriction_of_OR}.

    \emph{Lower bound.} Let $N = \{1, \dots, n\}$ be the first $n$ customers of the sequence $(X_i, \xi_i)_{i \geq 1}$. Any first-stage solution under the ECCs partitions $N$ into routes of at most $\kappa$ customers each, and is thus feasible for the CVRP of Section \ref{sec:asymptotics:sdvrp} for locations $(X_i)_{i = 1}^n$. Together with the non-negativity of the recourse costs under either policy $\pi \in \{\text{DTD}, \text{OR}\}$, this gives $V^{E, \pi}_n(X) \geq V^{\text{CVRP}}_n(X)$. The first inequality of \eqref{eq:ECC_rate} then follows from Theorem~\ref{thm:haimovich_cvrp}.

    \emph{Upper bound.} Consider again the first $n$ customers of the sequence $(X_i, \xi_i)_{i \geq 1}$, and let $\{(0, p^k, 0)\}_{k=1}^K$ be an optimal CVRP solution for locations $(X_i)_{i=1}^n$, with $p^k = (i^k_1, \dots, i^k_{t_k})$ and $t_k \leq \kappa$. Since this solution satisfies the ECCs, we have:
    \begin{equation}\label{eq:CVRP-DTD}
        V^{E, \text{DTD}}_n(X) \leq V^{\text{CVRP}}_n(X) + \sum_{k=1}^K \bar{\mathcal{Q}}^{\text{DTD}}_{p^k}.
    \end{equation}
    Applying Lemma~\ref{lem:recourse_bound} to each route $p^k$ and using $t_k\mu/Q \leq \kappa\mu/Q$ gives:
    \[
        \sum_{k=1}^K \bar{\mathcal{Q}}^{\text{DTD}}_{p^k} \leq \frac{\kappa\mu}{Q} \sum_{k=1}^K L(p^k) = \frac{\kappa\mu}{Q} \, V^{\text{CVRP}}_n(X),
    \]
    so $V^{E, \text{DTD}}_n(X) \leq V^{\text{CVRP}}_n(X) \cdot (1 + \kappa\mu/Q)$. Applying Theorem~\ref{thm:haimovich_cvrp} and using $\kappa\mu/Q = 1/\alpha$ yields the last inequality of \eqref{eq:ECC_rate} for $\pi = \text{DTD}$.
\end{proof}

In Proposition~\ref{prop:ECC_tightness}, we exhibit families of distributions of locations and demands that approach each bound of Theorem~\ref{thm:ECC_rate} arbitrarily closely, for any bin-packing constant $\alpha \in [1,2)$.

\begin{proposition}\label{prop:ECC_tightness}
    Consider the setup of Section~\ref{sec:asymptotics:setting}, and fix $\mu \leq Q$ with associated per-route customer cap $\kappa = \lfloor Q/\mu \rfloor$ and bin-packing constant $\alpha = Q/(\mu\kappa)$. Let customers be colocated at a point $x^* \in \mathbb{R}^2$ with $\|x^*\| > 0$, so $X_i = x^*$ with probability one, and $\bar d = \|x^*\|$. Under each policy $\pi \in \{\text{DTD}, \text{OR}\}$:
    \begin{enumerate}
        \item[(i)] Under deterministic demands $\xi_i = \mu$,
        \begin{equation}\label{eq:tightness_lb}
            \lim_{n \to \infty} \frac{V^{E, \pi}_n(X)}{n} = \alpha \cdot \frac{2\mu\bar d}{Q} \qquad \text{a.s.}
        \end{equation}
        \item[(ii)] For every $\varepsilon \in (0, Q/\kappa]$, under i.i.d.\ demands $\xi_i \sim (Q+\varepsilon) \cdot \mathrm{Bern}\bigl(\mu/(Q+\varepsilon)\bigr)$,
        \begin{equation}\label{eq:tightness_ub}
            \lim_{n \to \infty} \frac{V^{E, \pi}_n(X)}{n} = \alpha \cdot \frac{2\mu\bar d}{Q} + \frac{2\mu\bar d}{Q + \varepsilon} \qquad \text{a.s.},
        \end{equation}
        approaching $(\alpha + 1) \cdot 2\mu\bar d/Q$ as $\varepsilon \downarrow 0$.
    \end{enumerate}
\end{proposition}

\begin{proof}
    With customers colocated at $x^*$, any route has length $2\|x^*\|$ and any restocking trip costs $2\|x^*\|$.

    \paragraph{(i)} Cumulative demand on any ECC-feasible route of $t \leq \kappa$ customers is $t\mu \leq Q$, so no restocking is required. As the feasible solutions of the ECC-VRPSD and CVRP coincide, $V^{E, \pi}_n(X) = V^{\text{CVRP}}_n(X)$, and~\eqref{eq:tightness_lb} follows from Theorem~\ref{thm:haimovich_cvrp}.

    \paragraph{(ii)} Consider any ECC-feasible first-stage solution $\{(0, p^k, 0)\}_{k=1}^{K_n}$ with $p^k = (i^k_1, \dots, i^k_{t_k})$, $t_k \leq \kappa$. For each route $p^k$, let $M_k$ denote the number of successes ($\xi_{i^k_j} = Q+\varepsilon$) along $p^k$, and let $R^\pi_k$ denote the number of restocking trips along $p^k$ under policy $\pi \in \{\text{DTD}, \text{OR}\}$. Since each success exceeds the vehicle's capacity $Q$, at least one restock is required per success, giving $R^\pi_k \geq M_k$ pathwise under any policy. Conversely, $\varepsilon \leq Q/\kappa$ ensures that before the $m$-th success along a route, $m \in \{1,\dots, \kappa\}$, the residual capacity is $Q - (m-1)\varepsilon \geq Q - (\kappa-1)\varepsilon \geq \varepsilon$, so each success can be served with exactly one restock, hence $R^{\text{DTD}}_k = M_k$. Since all restocks cost the same and the OR policy is cost-minimizing, $\mathbb{E}[R^{\text{OR}}_k] \leq \mathbb{E}[R^{\text{DTD}}_k] = t_k \mu/(Q+\varepsilon)$; combined with $\mathbb{E}[R^\pi_k] \geq \mathbb{E}[M_k] = t_k \mu/(Q+\varepsilon)$, both policies satisfy $\mathbb{E}[R^\pi_k] = t_k \mu/(Q+\varepsilon)$. The expected total cost of a solution comprising $K_n$ routes is thus:
\[
    \sum_{k=1}^{K_n} \bigl(2\|x^*\| + \mathbb{E}[R^\pi_k] \cdot 2\|x^*\|\bigr) = 2 K_n \|x^*\| + \frac{2\mu \|x^*\| n}{Q+\varepsilon},
\]
and is minimized at the smallest ECC-feasible number of routes, which is $K_n = \lceil n/\kappa \rceil$. Dividing by $n$ and using $\lceil n/\kappa \rceil/n \to 1/\kappa$ establishes \eqref{eq:tightness_ub}:
\[
    \lim_{n \to \infty} \frac{V^{E, \pi}_n(X)}{n} = \frac{2\|x^*\|}{\kappa} + \frac{2\mu\|x^*\|}{Q+\varepsilon} = \alpha \cdot \frac{2\mu\bar d}{Q} + \frac{2\mu\bar d}{Q+\varepsilon}.
\]
\end{proof}

Theorems~\ref{thm:asymptotic} and~\ref{thm:ECC_rate} together characterize how imposing the ECCs inflates the per-customer cost in the asymptotic regime. Combined with Proposition~\ref{prop:ECC_tightness}, this yields a tight value for the worst-case asymptotic ratio over admissible joint distributions of locations and demands.

\begin{corollary}\label{cor:ECC_worst_case}
Let $\mathcal{D}$ denote the set of joint distributions of $(X,\xi)$ satisfying the assumptions of Section~\ref{sec:asymptotics:setting} with $\mu \leq Q$. Each policy $\pi \in \{\text{DTD}, \text{OR}\}$ satisfies:
\begin{equation}\label{eq:ECC_worst_case}
    \sup_{P \in \mathcal{D}} \, \limsup_{n \to \infty} \frac{V^{E,\pi}_n(X)}{V^\pi_n(X)} = 3,
\end{equation}
where, for each $P \in \mathcal{D}$, the inner limit supremum is taken in the a.s. sense along an i.i.d.\ sequence $(X_i,\xi_i)_{i\geq1}$ with law $P$.
\end{corollary}

\begin{proof}
Fix a distribution $P \in \mathcal{D}$, and let $\alpha \in [1,2)$ be the resulting bin-packing constant. By Theorems~\ref{thm:asymptotic} and~\ref{thm:ECC_rate},
\[
    \limsup_{n \to \infty} \frac{V^{E,\pi}_n(X)}{V^\pi_n(X)} \le \alpha + 1 < 3
    \qquad \text{a.s.}
\]
Hence the supremum in~\eqref{eq:ECC_worst_case} is at most $3$. This bound is approached by the colocated locations and Bernoulli demands of Proposition~\ref{prop:ECC_tightness}(ii) with parameters $(\mu,\varepsilon)$: combining the limit therein with Theorem~\ref{thm:asymptotic} gives:
\[
\lim_{n\to\infty}\frac{V^{E,\pi}_n(X)}{V^\pi_n(X)}
=
\alpha + \frac{Q}{Q+\varepsilon}
\qquad \text{a.s.},
\]
where $\alpha = Q/(\mu \lfloor Q/\mu \rfloor)$. Taking $\mu \downarrow Q/2$ yields $\alpha \uparrow 2$, and then letting $\varepsilon \downarrow 0$ yields $\alpha + Q/(Q+\varepsilon) \uparrow 3$.
\end{proof}

With Corollary~\ref{cor:ECC_worst_case}, we partially resolve an open question posed by \cite{hoogendoorn2025evaluation}. Following the notation and assumptions of Section \ref{sec:ProbForm}, call an \emph{instance} of the VRPSD a choice of graph $(N_0,E)$, traveling costs $\{c_e\}_{e \in E}$, vehicle capacity $Q$, and demand distributions $\{\xi_i\}_{i \in N}$. Denote by $\mathcal{I}^E$ the set of all such instances where each customer's expected demand does not exceed the vehicle capacity, i.e., $\max_{i \in N}\{\mu_i\} \leq Q$, so that the ECC-VRPSD admits a feasible solution. For an instance $I \in \mathcal{I}^E$, denote by $V^{E,\pi}(I)$ and $V^{\pi}(I)$ the optimal values of the ECC and basic variants under policy $\pi$, and let $r^{\pi}(I) := V^{E,\pi}(I) / V^{\pi}(I)$ be the ratio between these optimal values. \cite{hoogendoorn2025evaluation} investigate the worst-case ratio $\sup_{I \in \mathcal{I}^E} r^{\pi}(I)$ and, in their Theorem 5, exhibit a family of instances $(I_n)_{n \geq 1}$ with $n$ colocated customers for which $\lim_{n \to \infty} r^{\pi}(I_n) = 3$ for both recourse policies $\pi \in \{\text{DTD}, \text{OR}\}$, establishing the lower bound $\sup_{I \in \mathcal{I}^E} r^{\pi}(I) \geq 3$. They establish no finite upper bound on this supremum, leaving open the possibility for it to be strictly greater than 3 or even infinite, and identify the determination of the worst-case ratio as an open question.

Corollary~\ref{cor:ECC_worst_case} resolves this question in our asymptotic regime: for customers drawn i.i.d.\ in a compact region of the Euclidean plane, with each customer's demand independent of its location, the worst-case limiting ratio is exactly~$3$. The question of whether $\sup_{I \in \mathcal{I}^E} r^{\pi}(I)=3$ holds for arbitrary instances outside the asymptotic i.i.d.\ regime remains open.
}

\section{Conclusion} \label{sec:Conc}

In this paper, we developed a comprehensive analysis of the DL-shaped method for the VRPSD, showing that superadditivity of the recourse function under path concatenation is the structural property that characterizes its validity. We established that this property holds under suitable conditions for the DTD policy, and without further assumptions for the OR policy. We then introduced the edge-set cuts, a new family of valid inequalities that generalize the original DL-shaped cuts \rev{and provide structural advantages over existing valid inequalities for the VRPSD}. Our implementation of the DL-shaped method under the OR policy \rev{achieves state-of-the-art results on instances comprising a small number of long routes, the regime in which B\&C is the leading exact approach}. \rev{We complemented these algorithmic contributions with an asymptotic analysis of the optimal value of the VRPSD. We established that the VRPSD and the SDVRP have the same asymptotic per-customer cost, and characterized the worst-case impact of imposing the expected capacity constraints.}

Future work on the DL-shaped method could focus on developing \rev{other edge-selection strategies for the E-cuts, possibly building on the relation to partial route-split inequalities established in Theorem~\ref{thm:ecut-prs}. A more ambitious direction is the development of hybrid decomposition strategies leveraging the complementarity of the B\&C and BP\&C methods}. Our results for the single-route variant of the VRPSD could also motivate the development of DL-shaped methods for stochastic variants of the TSP. \rev{Finally, relaxing the structural assumptions of our framework, by investigating the extent to which superadditivity persists under demand correlation, or by extending the superadditivity property to non-route-based recourse paradigms, defines a natural direction for broadening the applicability of the DL-shaped methodology.}

\section*{Acknowledgement} The authors gratefully acknowledge the support of the Natural Sciences and Engineering Research Council of Canada (NSERC) [grants 2021-04037 and 2019-00094]. We thank the Digital Research Alliance of Canada for providing high-performance parallel computing facilities.

\bibliographystyle{informs2014}
\bibliography{refs_abbv}

\end{document}